\documentclass[10pt,oneside,a4paper]{article}
\usepackage{amsmath} 
\usepackage{amsfonts} 
\usepackage{amssymb} 
\usepackage{amsthm} 
\usepackage{hyperref}

\usepackage{stmaryrd} 
\usepackage{accents} 

\usepackage{xcolor} 
\usepackage[normalem]{ulem} 
\usepackage{soul} 

\usepackage[algo2e,vlined,ruled,algosection,linesnumberedhidden]{algorithm2e}

\usepackage{multirow} 
\usepackage{booktabs} 

\usepackage{graphicx}
\usepackage{subcaption} 

\usepackage{tikz} 
\usetikzlibrary{positioning}

\newcommand{\Be}{\boldsymbol{e}}

\newcommand{\Bg}{\boldsymbol{g}}

\newcommand{\Bu}{\boldsymbol{u}}
\newcommand{\Bv}{\boldsymbol{v}}

\newcommand{\Bx}{\boldsymbol{x}}

\newcommand{\hBx}{\boldsymbol{\hat x}}

\newcommand{\BE}{\boldsymbol{E}}

\newcommand{\BH}{\boldsymbol{H}}

\newcommand{\BL}{\boldsymbol{L}}
\newcommand{\BM}{\boldsymbol{M}}
\newcommand{\BN}{\boldsymbol{N}}

\newcommand{\BU}{\boldsymbol{U}}
\newcommand{\BV}{\boldsymbol{V}}

\newcommand{\BX}{\boldsymbol{X}}

\newcommand{\bbC}{\mathbb{C}}

\newcommand{\bbE}{\mathbb{E}}

\newcommand{\bbN}{\mathbb{N}}

\newcommand{\bbR}{\mathbb{R}}

\newcommand{\calK}{\mathcal{K}}
\newcommand{\calL}{\mathcal{L}}
\newcommand{\calM}{\mathcal{M}}

\newcommand{\calP}{\mathcal{P}}
\newcommand{\calQ}{\mathcal{Q}}

\newcommand{\calS}{\mathcal{S}}
\newcommand{\calT}{\mathcal{T}}

\newcommand{\calX}{\mathcal{X}}
\newcommand{\calY}{\mathcal{Y}}
\newcommand{\calZ}{\mathcal{Z}}

\newcommand{\Beta}{\boldsymbol{\eta}}

\newcommand{\Bnu}{\boldsymbol{\nu}}
\newcommand{\Bxi}{\boldsymbol{\xi}}

\newcommand{\Bphi}{\boldsymbol{\phi}}
\newcommand{\Bvarphi}{\boldsymbol{\varphi}}

\newcommand{\Bpsi}{\boldsymbol{\psi}}

\newcommand{\rn}[1]{\uppercase\expandafter{\romannumeral #1}}

\newcommand{\nn}{\nonumber}
\newcommand{\be}{\begin{eqnarray}}
\newcommand{\ee}{\end{eqnarray}}
\newcommand{\ben}{\begin{eqnarray*}}
\newcommand{\een}{\end{eqnarray*}}

\newtheorem{theorem}{Theorem}[section]
\newtheorem{lemma}[theorem]{Lemma}

\newtheorem{definition}[theorem]{Definition}

\numberwithin{equation}{section}
\numberwithin{theorem}{section}

\numberwithin{table}{section}
\numberwithin{figure}{section}

\newcommand{\bs}{\backslash}

\newcommand{\N}[1]{\left\|{#1}\right\|} 
\newcommand{\abs}[1]{\left|{#1}\right|} 

\newcommand{\rmi}{\mathrm{i}}

\newcommand{\rmd}{\mathrm{d}}

\providecommand{\Div}{\operatorname{Div}} 

\providecommand{\curl}{\operatorname{\mathbf{curl}}} 
\providecommand{\Curl}{\operatorname{Curl}} 

\renewcommand{\Re}{\operatorname{Re}} 
\renewcommand{\Im}{\operatorname{Im}} 

\newcommand{\Lpv}[2][\defaultdomain]{\BL^{#2}({#1})}

\newcommand{\NLpv}[3][\defaultdomain]{\N{#2}_{\Lpv[#1]{#3}}}

\newcommand{\Ltwov}[1][\defaultdomain]{\Lpv[#1]{2}}

\newcommand{\NLtwov}[2][\defaultdomain]{\NLpv[#1]{#2}{2}}

\newcommand{\Hcurl}[1][\defaultdomain]{\boldsymbol{H}(\curl,{#1})}
\newcommand{\bHcurl}[2][\defaultdomain]{\boldsymbol{H}_{#2}(\curl,{#1})}

\makeatletter
\renewcommand{\@biblabel}[1]{#1.}
\makeatother

\begin{document}

\title{Edge element DtN method for electromagnetic scattering poles of perfectly conducting obstacles}

\author{Bo Gong$^1$, Takumi Sato$^2$, Jiguang Sun$^3$, Xinming Wu$^{2,}$\thanks{Corresponding author. Email address: wuxinming@fudan.edu.cn}}

\footnotetext[1]{School of Mathematics, Statistics and Mechanics, Beijing University of Technology, Beijing, 100124, China.}
\footnotetext[2]{School of Mathematical Sciences, SKLCAM, Fudan University, Shanghai 200433, China.}
\footnotetext[3]{Department of Mathematical Sciences, Michigan Technological University, Houghton, MI 49931, U.S.A.}

\date{}

\maketitle

\begin{abstract}
Meromorphic continuation of the scattering operator leads to scattering poles (resonances) in the complex plane. Despite their significance, numerical investigation  of scattering poles remains limited. In this paper, we propose and analyze a numerical method to compute electromagnetic poles of perfectly conducting obstacles. The unbounded domain for the scattering problem is truncated using the DtN mapping and the poles are shown to be the eigenvalues of a holomorphic Fredholm operator function related to Maxwell's equations. Edge elements are used for discretization. The convergence is proved using the abstract approximation theory for eigenvalue problems of holomorphic Fredholm operator functions. The proposed finite element DtN approach is free of non-physical poles. A spectral indicator method is then employed to compute the resulting nonlinear matrix eigenvalue problem. Numerical examples are presented to demonstrate the effectiveness of the method.
\end{abstract}

\newenvironment{keywords}{\par\noindent\textbf{Keywords}}{}


\section{Introduction}

Scattering resonances, governing energy decay, field confinement, and spectral response, etc., play a fundamental role in wave applications such as antenna design, acoustic resonators, metallic grating structures, and inverse spectral problems \cite{Baum86, LaxPhillips1989, Huang, Cakoni2020ProcA}. Mathematically, scattering resonances are the poles of the meromorphic continuation of the scattering operator in the complex plane \cite{Zworski1999, Taylor1996, DyatlovZworski2019, Uberall83}. In this paper, we do not distinguish resonances and poles, although the definition of resonances can be different depending on context. 

Despite their significance, systematic investigation on resonance computation remains limited. For perfectly conducting
obstacles, scattering resonances are  complex with negative imaginary parts, and their associated eigenfunctions grow exponentially away from the scatterer. The development of numerical methods for scattering resonances faces delicate challenges such as spectral pollution, domain truncation, and nonlinearity of discrete systems. Furthermore, error analysis of computational methods for nonlinear spectral problems is largely undeveloped, limiting both theoretical understanding and computational reliability. 


There exist two main groups of numerical methods for scattering poles. The first group relies on boundary integral operators, which inherently satisfy the outgoing condition \cite{SteinbachUnger2012, OlafUnger2017MMAS, Unger21, MaSun2023, Liu2025WM, Matsushima2025ProcA}. These methods require discretization only on the boundary of the scatterer. However, standard boundary integral formulations introduce non-physical resonances which can be difficult to distinguish (see \cite{Unger21, Hiptmair2022, Grubisic2023}). Complicate structures and non-homogeneous materials poses additional difficulties.

The second group comprises finite element methods, which typically require the truncation of the unbounded domain. The Dirichlet-to-Neumann (DtN) mapping and the Perfectly Matched Layer (PML) have been employed \cite{Lenoir1992, KimPasciak2009MC, Kim2014, NannenWess2018BIT, Araujo2021JCP, HohageNannen2009, Halla2022SINUM, XiLinSun2024CMA, XiGongSun2024, Gong2026AML}. While (frequency-dependent) PML results in a linear eigenvalue problem, it introduces spurious eigenvalues that are difficult to identify. In addition, since the eigenfunctions increase exponentially away from the scatterer, to force them decay in the PML layer, the choice of PML parameters presents a significant challenge. We refer the readers to \cite{KimPasciak2009MC, NannenWess2018BIT} for more discussion. In contrast, DtN mapping does not contaminate the spectrum and is particularly advantageous with the development of highly efficient and robust contour integral-based methods for nonlinear matrix eigenvalue problems \cite{XiGongSun2024, XiLinSun2024CMA, XiSun2023, Gong2026AML, SunZhou2016}. It is worth noting that Hardy space infinite elements have also been investigated \cite{HohageNannen2009, Nannen2013SISC}.


In this paper, we propose and analyze the edge element DtN method to compute the scattering poles of a perfectly conducting  obstacle. The scattering poles are formulated as the eigenvalues of a holomorphic Fredholm operator function defined by the variational formulation related to the scattering problem. The unbounded domain is truncated and the DtN mapping is employed, leading to an equivalent problem posed on the bounded domain (free of non-physical poles). The convergence is proved using the abstract spectral approximation theory of holomorphic Fredholm operator functions \cite{Karma1, Karma2} by taking into account both errors of the truncated DtN mapping and the edge element discretization. Finally, the parallel spectral indicator method (SIM) is used to compute all the poles in a region on the lower half complex plane \cite{XiSun2023}. The main contributions are as follows. 1) We construct a holomorphic Fredholm operator function using the DtN mapping such that its eigenvalues coincide with the scattering poles on the continuous level; 2) Error estimates of the edge element DtN method are obtained, which is considerably more technical than the acoustic case. The analysis is carried out using curl-conforming function spaces and avoids the need for a Helmholtz decomposition;
3) We present several numerical examples, which demonstrate the theory and serve as benchmarks for future investigation as there exist few full 3D electromagnetic simulations.

The rest of the paper is organized as follows. In Section~\ref{Sec2}, we introduce the electromagnetic scattering problem due to a perfectly conducting obstacle and truncate the unbounded domain using the DtN mapping. We formulate the scattering poles as the eigenvalues of a holomorphic Fredholm operator function. Property of the truncated DtN operator is also obtained. In Section~\ref{Sec3}, we discretize the operator function using the edge element  and prove the convergence of the eigenvalues. Numerical examples are presented in Section~\ref{Sec4}.


\section{Scattering poles of PEC obstacles}\label{Sec2}

We start with the time-harmonic electromagnetic scattering by a perfect electric conducting (PEC) obstacle.
Let $D\subset \bbR^3$ be a bounded Lipschitz 
domain with  boundary $\Gamma_D$ and let $\Bnu$ be the unit outer normal to $\Gamma_D$.
Given a known incident electromagnetic field $(\BE^{\rm inc}, \BH^{\rm inc})$, the total electromagnetic field $(\BE, \BH)$ is governed by Maxwell's equations
\begin{subequations}
\begin{align}
&\curl\BE - \rmi\kappa\BH = 0 &&\text{in}\ \ \bbR^3\bs\bar D, \label{eq:em1}\\
&\curl\BH + \rmi\kappa\BE = 0 &&\text{in}\ \ \bbR^3\bs\bar D, \label{eq:em2}\\
&\Bnu\times\BE = 0 &&\text{on}\ \ \Gamma_D, \label{eq:em3}\\
&|\Bx|\, (\BH^s\times\hBx - \BE^s) \to 0 &&\text{as}\ \  |\Bx| \to +\infty, \label{eq:em4}
\end{align}
\end{subequations}
where $\kappa$ is the wave number,  $\BE^s = \BE - \BE^{\rm inc}$ and $\BH^s = \BH - \BH^{\rm inc}$ are the scattered electric and magnetic fields, respectively. Equation \eqref{eq:em3} is the PEC boundary condition and equation \eqref{eq:em4} is the Silver-M\"uller radiation condition.
Eliminating the magnetic field $\BH$ from \eqref{eq:em1}-\eqref{eq:em2}, we obtain the scattering problem for the electric field
\begin{subequations}
\begin{align}
&\curl\curl\BE - \kappa^2\BE = 0 &&\text{in}\ \ \bbR^3\bs\bar D, \label{eq:e1}\\
&\Bnu\times\BE = 0 &&\text{on}\ \ \Gamma_D, \label{eq:e2}\\
&|\Bx| (\curl\BE^s\times\hBx - \rmi \kappa\BE^s) \to 0\ &&\text{as}\ \  |\Bx| \to +\infty.\label{eq:e3}
\end{align}
\end{subequations}
For $\Im(\kappa)\ge 0$, the scattering problem \eqref{eq:e1}-\eqref{eq:e3} has a unique solution $\BE\in \bHcurl[\bbR^3\bs\bar D]{\rm loc}$ (c.f. ~\cite{CK13, CK19, Monk03}). The solution operator function $\calS(\kappa): \BE^{\rm inc}\to \BE$ is holomorphic on the upper half complex plane and can be meromorphically extended to $\bbC\bs\bbR^-$ with $\bbR^-$ being the set of non-positive real numbers. The scattering resonances are the poles of $\calS(\kappa)$ in $\{\kappa\in\bbC: \Im(\kappa) < 0\}$. 


We now discuss the outgoing condition for the scattered fields, which is equivalent to \eqref{eq:e3} for $\kappa$ with $\Im(\kappa)\ge 0$ and suitable for $\kappa$ with $\Im(\kappa)< 0$.
Let $Y_n^m(\hBx),\, m=-n,\cdots,n,\,n=0,1,\cdots$ denote an orthonormal sequence of spherical harmonics on the unit sphere $\partial B_1$ such that 
\begin{equation*}
\Delta_{\partial B_1} Y_n^m(\hBx) + n(n+1) Y_n^m(\hBx) = 0\quad\text{on}\ \ \partial B_1, 
\end{equation*}
where
\begin{equation*}
\Delta_{\partial B_1}  = \frac{1}{\sin\theta}\frac{\partial}{\partial\theta}\left(\sin\theta\frac{\partial }{\partial\theta}\right)
+ \frac{1}{\sin^2\theta}\frac{\partial^2 }{\partial\phi^2}
\end{equation*}
is the Laplace-Beltrami operator.
Define the vector spherical harmonics by
\begin{equation}
\BU_n^m = \frac{1}{\sqrt{n(n+1)}}\nabla_{\partial B_1} Y_n^m,\qquad \BV_n^m = \hBx\times\BU_n^m
\label{eq:vec-sph-harm}
\end{equation}
for $n=1,2,\cdots$ and $m=-n,\cdots,n$, where
\begin{eqnarray*}
\nabla_{\partial B_1}  = \frac{\partial }{\partial\theta}\Be_\theta + \frac{1}{\sin\theta}\frac{\partial }{\partial\phi}\Be_\phi
\end{eqnarray*}
is the surface gradient operator for $\partial B_1$ and $\{\Be_r,\Be_\theta,\Be_\phi\}$ are the unit vectors of the spherical coordinates.
The set of all vector spherical harmonics defined in \eqref{eq:vec-sph-harm} forms a complete orthonormal basis for $\BL^2_t(\partial B_1) = \{\Bpsi\in\Ltwov[\partial B_1]:\, \Bpsi\cdot\hBx = 0\}$.

Let $h_n^{(1)}(z)$ be the spherical Hankel function of the first kind of order $n$. Define the vector wave functions
\begin{equation*}
\BM_n^m(r,\hBx) = \curl\{\Bx h_n^{(1)}(\kappa r) Y_n^m(\hBx)\},\quad \BN_n^m(r,\hBx) = \frac{1}{\rmi\kappa} \curl\BM_n^m(r,\hBx)
\end{equation*}
for $n=1,2,\cdots$ and $m=-n,\cdots,n$. 
Since the function $h_n^{(1)}(\kappa r) Y_n^m(\hBx)$ satisfies the Helmholtz equation,  it can be shown that the vector functions $\BM_n^m(r,\hBx)$ and $\BN_n^m(r,\hBx)$ also constitute solutions to Maxwell's equations in $\bbR^3\bs\{0\}$. Moreover, when $\Im(\kappa)\ge 0$ these functions satisfy the Silver-M\"uller condition (c.f.~\cite{Monk03}).

Let the obstacle $D$ be contained in the ball $B_R=\{\Bx\in\bbR^3\, :\, |\Bx|<R\}$ with the boundary $\Gamma_R$.
It is well-known that, for $\Im(\kappa)\ge 0$, the radiation condition \eqref{eq:e3} for $\BE^s$ is equivalent to the series expansion of the form
\begin{equation}
\BE^s(r,\hBx) =  \sum_{n=1}^{\infty}\sum_{m=-n}^n \alpha_n^m\BM_n^m(r,\hBx) + \beta_n^m\BN_n^m(r,\hBx),\quad r=|\Bx| > R.
\label{eq:outgoing}
\end{equation}
We say $\BE^s$ is outgoing if it satisfies \eqref{eq:outgoing} and refer \eqref{eq:outgoing} as the outgoing condition. Since poles have negative imaginary part, the Silver-M\"uller radiation condition \eqref{eq:e3} no longer characterizes the outgoing waves and \eqref{eq:outgoing} is the correct condition to use.

\subsection{Transparent boundary condition}

To reduce the scattering problem \eqref{eq:e1}-\eqref{eq:e3} on a bounded domain, we introduce a transparent boundary condition (TBC). Note that 
\begin{align*}
\BM_n^m(r,\hBx) &= h_n^{(1)}(\kappa r) \nabla_{\partial B_1} Y_n^m(\hBx) \times \hBx,\\
\BN_n^m(r,\hBx) &= \frac{z_n^{(1)}(\kappa r) }{\rmi\kappa r} \nabla_{\partial B_1} Y_n^m(\hBx) 
+\frac{n(n+1)}{\rmi\kappa r} h_n^{(1)}(\kappa r) Y_n^m(\hBx) \hBx ,
\end{align*} 
where $z_n^{(1)}(z) = h_n^{(1)}(z) + zh_n^{(1)'}(z)$.
Then by the definition of $\BU_n^m$ and $\BV_n^m$, we have
\begin{align}
\BM_n^m\times\hBx  &= -\sqrt{n(n+1)} h_n^{(1)}(\kappa r) \BU_n^m(\hBx),
\label{eq:Mtx}\\
\BN_n^m\times\hBx  &= -\sqrt{n(n+1)} \frac{z_n^{(1)}(\kappa r)}{\rmi\kappa r} \BV_n^m(\hBx).
\label{eq:Ntx}
\end{align}
Denote by $\BE^s_T=\gamma_T\BE^s=\hBx\times(\BE^s\times\hBx)$ the tangential component of $\BE^s$ on $\Gamma_R$. From \eqref{eq:outgoing} and \eqref{eq:Mtx}-\eqref{eq:Ntx} we obtain that
\begin{equation}
\BE^s_T = \sum_{n=1}^{\infty}\sum_{m=-n}^n  \sqrt{n(n+1)} \left( -\alpha_n^m h_n^{(1)}(\kappa R) \BV_n^m(\hBx) 
+ \beta_n^m \frac{z_n^{(1)}(\kappa R)}{\rmi\kappa R} \BU_n^m(\hBx)\right)
\label{eq:EDir}
\end{equation}
and
\begin{align}
&\frac{1}{\rmi\kappa}\curl\BE^s\times\hBx =  \sum_{n=1}^{\infty}\sum_{m=-n}^n 
\alpha_n^m (\BN_n^m\times\hBx ) - \beta_n^m (\BM_n^m\times\hBx) \nn\\
=& \sum_{n=1}^{\infty}\sum_{m=-n}^n  \sqrt{n(n+1)} \left( -\alpha_n^m \frac{z_n^{(1)}(\kappa R)}{\rmi\kappa R} \BV_n^m(\hBx)
+\beta_n^m h_n^{(1)}(\kappa R) \BU_n^m(\hBx)\right).
\label{eq:ENeum}
\end{align}

Define the Sobolev spaces
\begin{align*}
&\BH^{s}(\Curl, \Gamma_R) =  \{\Bxi\in\BH^{s}_t(\Gamma_R):\,\Curl\Bxi\in H^{s}(\Gamma_R)\},\\
&\BH^{s}(\Div,\Gamma_R)  =  \{\Bxi\in\BH^{s}_t(\Gamma_R):\,\Div\Bxi\in H^{s}(\Gamma_R)\},
\end{align*}
where $\Curl$ and $\Div$ are the surface curl and surface divergence operators on $\Gamma_R$, and $\BH^{s}_t(\Gamma_R) =\BH^{s}(\Gamma_R)\cap\BL^2_t(\Gamma_R)$.
Given a tangential vector
\begin{equation}
\Bxi = \sum_{n=1}^{\infty}\sum_{m=-n}^n a_n^m\BU_n^m + b_n^m\BV_n^m\quad\text{on}\ \Gamma_R,
\label{eq:xi}
\end{equation}
the norm on the space $\BH^{s}_t(\Gamma_R)$ can be characterized by
\begin{equation}
\label{eq:norm-Hs}
\N{\Bxi}_{\BH^{s}_t(\Gamma_R)}^2 = \sum_{n=1}^{\infty}\sum_{m=-n}^n (n(n+1))^s (\abs{a_n^m}^2 + \abs{b_n^m}^2).
\end{equation}
The norms on the spaces $\BH^{s}(\Curl, \Gamma_R)$ and $\BH^{s}(\Div, \Gamma_R)$ can be respectively defined by \cite{Monk03}
\begin{align}
&\N{\Bxi}_{\BH^{s}(\Curl, \Gamma_R)}^2 = \sum_{n=1}^{\infty}\sum_{m=-n}^n (n(n+1))^{s} \abs{a_n^m}^2 + (n(n+1))^{s+1} \abs{b_n^m}^2,
\label{eq:norm-HsCurl}\\
&\N{\Bxi}_{\BH^{s}(\Div, \Gamma_R)}^2 = \sum_{n=1}^{\infty}\sum_{m=-n}^n (n(n+1))^{s+1} \abs{a_n^m}^2 + (n(n+1))^{s} \abs{b_n^m}^2.
\label{eq:norm-HsDiv}
\end{align}

The Calder\'on operator $\calT(\kappa): \BH^{-\frac12}(\Curl, \Gamma_R)\to\BH^{-\frac12}(\Div,\Gamma_R)$ is the Dirichlet-to-Neumann operator defined by 
\begin{equation}
\calT(\kappa)\Bxi= \frac{1}{\rmi\kappa}\curl\BE^s\times\hBx \quad \text{on}\ \ \Gamma_R,
\end{equation}
where $\BE^s=\bbE(\Bxi)$ is the extension of $\Bxi$ to $\bbR^3\bs\bar B_R$ which satisfies
\begin{subequations}
\begin{align}
&\curl\curl\BE^s - \kappa^2\BE^s = 0 &&\text{in}\ \ \bbR^3\bs\bar B_R,\\
&\BE^s_T = \Bxi &&\text{on}\ \ \Gamma_R,\\
&\BE^s \quad \text{is outgoing.}
\end{align}
\end{subequations}
By \eqref{eq:EDir}, the extension $\BE^s=\bbE(\Bxi)$ satisfies \eqref{eq:outgoing} with the coefficients
\begin{equation}
\alpha_n^m =  -\frac{b_n^m}{\sqrt{n(n+1)}h_n^{(1)}(\kappa R)},\quad
\beta_n^m = \frac{i\kappa R a_n^m}{\sqrt{n(n+1)}z_n^{(1)}(\kappa R)}.
\label{eq:Extension}
\end{equation}
Substitution of \eqref{eq:Extension} into \eqref{eq:ENeum} yields (see, e.g., Monk~\cite{Monk03} and Bao et al.~\cite{Bao23})
\begin{equation}
\calT(\kappa)\Bxi = \sum_{n=1}^{\infty}\sum_{m=-n}^n 
\frac{\rmi\kappa R \,h_n^{(1)}(\kappa R)}{z_n^{(1)}(\kappa R)}a_n^m\BU_n^m
+ \frac{z_n^{(1)}(\kappa R)}{\rmi\kappa R\,h_n^{(1)}(\kappa R)}b_n^m\BV_n^m.
\label{eq:Calderon}
\end{equation}
For simplicity, we introduce the notation $\delta_n(\kappa) = \frac{z_n^{(1)}(\kappa R)}{h_n^{(1)}(\kappa R)}$. Then \eqref{eq:Calderon} can be written as
\begin{equation}
\calT(\kappa)\Bxi = \sum_{n=1}^{\infty}\sum_{m=-n}^n 
\frac{\rmi\kappa R}{\delta_n(\kappa)} a_n^m\BU_n^m
+ \frac{\delta_n(\kappa)}{\rmi\kappa R} b_n^m\BV_n^m.
\label{eq:Calderon1}
\end{equation}

Let $\calZ_n$ be the set of zeros of $h_n^{(1)}(\kappa R)$ and $z_n^{(1)}(\kappa R)$ and $\calZ=\cup_{n=1}^\infty \calZ_n$. Define $\tilde\Lambda=\bbC\bs(\bbR^-\cup \calZ)$. For $\kappa\in\tilde\Lambda$, the Calder\'on operator $\calT(\kappa)$ can be defined by \eqref{eq:xi} and \eqref{eq:Calderon}.
$\calT(\kappa): \BH^{-\frac12}(\Curl, \Gamma_R)\to\BH^{-\frac12}(\Div,\Gamma_R)$  is bounded
\begin{equation}
\N{\calT(\kappa)\Bxi}_{\BH^{-\frac12}(\Div,\Gamma_R)} \le C \N{\Bxi}_{\BH^{-\frac12}(\Curl, \Gamma_R)},
\label{eq:Tbounded}
\end{equation}
where $C > 0$ is a constant depending on $\kappa R$ but not $\Bvarphi$ (\cite[Theorem 9.21]{Monk03}).
Using $\calT(\kappa)$, the scattering problem \eqref{eq:e1}-\eqref{eq:e3} can be formulated on the bounded domain $\Omega = B_R\bs\bar D$: 
\begin{subequations}
\begin{align}
&\curl\curl\BE - \kappa^2\BE = 0 &&\text{in}\ \ \Omega, \label{eq:eb1}\\
&\Bnu\times\BE = 0 &&\text{on}\ \ \Gamma_D, \label{eq:eb2}\\
&\curl\BE\times\Bnu - \rmi\kappa\calT(\kappa)\BE_T= \Bg &&\text{on}\ \ \Gamma_R, \label{eq:eb3}
\end{align}
\end{subequations}
where $\Bg = \curl\BE^{\rm inc}\times\Bnu -\rmi\kappa\calT(\kappa)\BE^{\rm inc}_T$.

\subsection{Weak formulation and operator equation}

Define the space
\begin{equation*}
\BX = \{ \Bvarphi\in \Hcurl\ |\ \Bnu\times\Bvarphi=0\ \text{on}\ \Gamma_D\},
\end{equation*}
equipped with the inner product
\begin{equation*}
(\Bvarphi, \Bpsi)_{\BX} = (\curl\Bvarphi, \curl\Bpsi) + (\Bvarphi, \Bpsi),\quad \Bu,\Bv \in \BX,
\end{equation*}
where $(\cdot,\cdot)$ denotes the $\Ltwov$ inner product. Given $\Bg\in \BH^{-\frac12}(\Div, \Gamma_R)$, the variational formulation for  \eqref{eq:eb1}-\eqref{eq:eb3} is to find $\BE\in \BX$ such that 
\begin{equation}
a_\kappa(\BE, \Bpsi) = \langle \Bg, \Bpsi \rangle_{\Gamma_R} \quad\forall\Bpsi\in \BX,
\label{eq:var}
\end{equation}
where the sesquilinear form $a_\kappa: \BX\times \BX \to\bbC$ is defined by
\begin{equation}
a_\kappa(\Bvarphi, \Bpsi) = (\curl\Bvarphi, \curl\Bpsi) - \kappa^2(\Bvarphi, \Bpsi) - \rmi\kappa\langle\calT(\kappa)\Bvarphi_T, \Bpsi_T\rangle_{\Gamma_R}.
\end{equation}
Here $\langle\cdot,\cdot\rangle_{\Gamma_R}$ denotes the $\BH^{-\frac12}(\Div, \Gamma_R)$-$\BH^{-\frac12}(\Curl,\Gamma_R)$ duality.

As the poles are complex with negative imaginary parts, to analyze the convergence, we fix a region in the lower complex plane 
\begin{equation}
\Lambda=\{\, \kappa=\kappa_1+\rmi\kappa_2\in\bbC\, |\, \abs{\kappa}\le\alpha_1, \, \kappa_2<-\alpha_2<0 \, \} \cap \tilde\Lambda,
\label{eq:Lambda}
\end{equation}
where $\alpha_1,\alpha_2$ are arbitrary positive constants.

Define the operators $B(\kappa): \BX \to \BX,\, \kappa\in\Lambda$ and $M: \BH^{-\frac12}(\Div, \Gamma_R)\to \BX$, respectively, by
\begin{equation}
(B(\kappa)\Bvarphi, \Bpsi)_{\BX}= a_\kappa(\Bvarphi, \Bpsi) \quad\forall \Bpsi\in \BX,
\label{eq:Bk}
\end{equation}
and
\begin{equation}
(M\Bg, \Bpsi)_{\BX} = \langle \Bg, \Bpsi_T\rangle_{\Gamma_R} \quad\forall \Bpsi\in \BX.
\label{eq:M}
\end{equation}
Then \eqref{eq:var} can be written as
\begin{equation}
B(\kappa)\BE = M\Bg,
\end{equation}
which defines the solution operator $B(\kappa)^{-1}M$ that maps $\Bg$ to $\BE$ for $\Im(\kappa) \ge 0$. The scattering resonances are the poles of $B(\cdot)^{-1}$, i.e., the eigenvalues of $B(\cdot)$. The eigenvalue problem for the operator function $B(\kappa)$ is to find $(\kappa, \BE)\in\Lambda\times \BX$, such that 
\begin{equation}
B(\kappa)\BE = 0\quad\text{in}\ \ \BX,
\label{eq:eig}
\end{equation}
or equivalently, 
\begin{equation}
a_\kappa(\BE, \Bpsi) = 0,\quad\forall\Bpsi\in \BX.
\label{eq:eig-weak}
\end{equation}

We now analyze the operator function $B(\kappa)$. Given the expansion of $\Bxi \in \BH^{-\frac12}(\Curl, \Gamma_R)$ as in \eqref{eq:xi}, we decompose the Calder\'on operator as $\calT(\kappa) = \calT_1(\kappa) + \calT_2(\kappa)$, where
\[
\calT_1(\kappa)\Bxi = \sum_{n=1}^{\infty}\sum_{m=-n}^n
\frac{\rmi\kappa R }{\delta_n(\rmi)} a_n^m \BU_n^m
+ \frac{\delta_n(\rmi) }{\rmi\kappa R} b_n^m\BV_n^m
\]
and
\[
\calT_2(\kappa)\Bxi = \sum_{n=1}^{\infty}\sum_{m=-n}^n 
\rmi\kappa R \left( \frac{1}{\delta_n(\kappa)} - \frac{1}{\delta_n(\rmi)} \right) a_n^m\BU_n^m
+ \frac{\delta_n(\kappa) - \delta_n(\rmi)}{\rmi\kappa R} b_n^m\BV_n^m.
\]
Accordingly, the sesquilinear form $a_\kappa$ is decomposed as $a_\kappa = a_\kappa^{(1)} + a_\kappa^{(2)}$, where
\begin{align}
&a_\kappa^{(1)}(\Bvarphi,\Bpsi) = (\curl\Bvarphi, \curl\Bpsi) - \kappa^2(\Bvarphi, \Bpsi) -\rmi\kappa \langle\calT_1(\kappa)\Bvarphi_T, \Bpsi_T\rangle_{\Gamma_R},\\
&a_\kappa^{(2)}(\Bvarphi,\Bpsi) = -\rmi\kappa \langle\calT_2(\kappa)\Bvarphi_T,\Bpsi_T\rangle_{\Gamma_R}.
\end{align}

\begin{lemma}\label{lem:inf-sup}
For $\kappa\in\Lambda$, the sesquilinear form $a_\kappa^{(1)}$ satisfies the following inf-sup condition
\begin{equation*}
\sup_{\Bpsi\in\BX}\frac{\abs{a_\kappa^{(1)}(\Bvarphi,\Bpsi)}}{\N{\Bpsi}_{\BX}} \ge C \N{\Bvarphi}_{\BX}\quad\forall\Bvarphi\in\BX,
\end{equation*}
where $C$ is a constant depending only on the constants $\alpha_1, \alpha_2$.
\end{lemma}

\begin{proof}
For any $\Bvarphi\in\BX$, let $\Bpsi=\overline{(\frac{1}{\rmi\kappa})}\Bvarphi$. This yields
\begin{equation*}
a_\kappa^{(1)}(\Bvarphi,\Bpsi) = \frac{1}{\rmi\kappa}\NLtwov{\curl\Bvarphi}^2 +\rmi\kappa\NLtwov{\Bvarphi}^2 - \langle\calT_1(\kappa)\Bvarphi_T,\Bvarphi_T\rangle_{\Gamma_R}.
\end{equation*}
For $\kappa=\kappa_1 + \rmi\kappa_2$, $\Re(\frac{1}{\rmi\kappa}) = -\frac{\kappa_2}{\abs{\kappa}^2}$, $\Re(\rmi\kappa) = -\kappa_2$. By \cite[Lemma 9.22]{Monk03}, $\delta_n(\rmi)$ is real and strictly negative for all $n$. Consequently,
\begin{equation*}
- \Re\langle\calT_1(\kappa)\Bvarphi_T,\Bvarphi_T\rangle_{\Gamma_R} \ge 0 \quad \text{for } \kappa_2 < 0.
\end{equation*}
It then follows that
\begin{equation*}
\abs{a_\kappa^{(1)}(\Bvarphi,\Bpsi)} \ge \Re\{{a_\kappa^{(1)}(\Bvarphi,\Bpsi)}\} \ge \min\left(1,\frac{1}{\abs{\kappa}^2}\right)\abs{\kappa_2}\N{\Bvarphi}_{\BX}^2,
\end{equation*}
and thus
\begin{equation*}
\sup_{\Bpsi\in\BX}\frac{\abs{a_\kappa^{(1)}(\Bvarphi,\Bpsi)}}{\N{\Bpsi}_{\BX}} \ge \min\left(\abs{\kappa},\frac{1}{\abs{\kappa}}\right)\abs{\kappa_2}\N{\Bvarphi}_{\BX}.
\end{equation*}
This completes the proof by taking $C=\min(\alpha_2^2,\alpha_2/\alpha_1)$.
\end{proof}

Notice that for large $n$, the function $h_n^{(1)}(z)$ behaves asymptotically as (c.f. Watson~\cite{Watson22})
\begin{equation}
h_n^{(1)}(z) = \frac{(2n-1)!!}{\rmi z^{n+1}}\left( 1+ O\left(\frac1n\right) \right),
\label{eq:hasymp}
\end{equation}
and hence
\begin{equation*}
z \frac{h_{n-1}^{(1)}(z)}{h_n^{(1)}(z)} = \frac{z^2}{2n-1}\left( 1+ O\left(\frac1n\right) \right).
\end{equation*}
Since the spherical Hankel function $h_n^{(1)}(z)$ has the recurrence formulation
\begin{equation}
h_n^{(1)'}(z) = h_{n-1}^{(1)}(z) - \frac{n+1}{z} h_n^{(1)}(z), 
\label{eq:hprime}
\end{equation}
it holds that
\begin{equation}
\frac{z_n^{(1)}(z)}{h_n^{(1)}(z)} = -n + z\frac{h_{n-1}^{(1)}(z)}{h_n^{(1)}(z)} = -n + \frac{z^2}{2n-1} + O\left(\frac{1}{n^2}\right). 
\label{eq:zdh}
\end{equation}

\begin{lemma}\label{lem:compact}
For $\kappa\in\Lambda$, the operator $\calT_2(\kappa)\gamma_T$, that is the mapping $\Bvarphi\to\calT_2(\kappa)\Bvarphi_T$ is compact from $\BX$ into $\BH^{-\frac12}(\Div, \Gamma_R)$.
\end{lemma}

\begin{proof}
For any $\Bvarphi\in\BX$, the trace $\Bxi=\Bvarphi_T\in\BH^{-\frac12}(\Curl, \Gamma_R)$. Given the expansion of $\Bxi$ as in \eqref{eq:xi}, we split $\Bxi=\Bxi^U+\Bxi^V$, and define
\begin{align*}
\calT_2^U(\kappa)\Bxi^U &= \sum_{n=1}^{\infty}\sum_{m=-n}^n 
\rmi\kappa R \left( \frac{1}{\delta_n(\kappa)} - \frac{1}{\delta_n(\rmi)} \right) a_n^m\BU_n^m,\\
\calT_2^V(\kappa)\Bxi^V &= \sum_{n=1}^{\infty}\sum_{m=-n}^n 
 \frac{\delta_n(\kappa) - \delta_n(\rmi)}{\rmi\kappa R} b_n^m\BV_n^m.
\end{align*}
By using the expansion \eqref{eq:zdh}, we have 
\begin{equation}
\delta_n(\kappa) - \delta_n(\rmi) = O\left(\frac1n\right),\quad
\frac{1}{\delta_n(\kappa)} - \frac{1}{\delta_n(\rmi)} = O\left(\frac{1}{n^3}\right).
\label{eq:zdh-diff}
\end{equation}
Then from the definitions of the norms in \eqref{eq:norm-Hs}-\eqref{eq:norm-HsDiv}, 
we obtain 
\begin{align*}
\N{\calT_2(\kappa)\Bxi}_{\BH^{-\frac12}(\Div, \Gamma_R)}^2 &= \N{\calT_2^U(\kappa)\Bxi^U}_{\BH^{\frac12}_t(\Gamma_R)}^2 + \N{\calT_2^V(\kappa)\Bxi^V}_{\BH^{-\frac12}_t(\Gamma_R)}^2\\
&\le C\left(\N{\Bxi_U}_{\BH^{-\frac52}_t(\Gamma_R)}^2 + \N{\Bxi_V}_{\BH^{-\frac32}_t(\Gamma_R)}^2\right) \le C\N{\Bxi}_{\BH^{-\frac52}(\Curl, \Gamma_R)}^2.
\end{align*}
Consequently, the operator $\calT_2(\kappa)$ is compact from $\BH^{-\frac12}(\Curl, \Gamma_R)$ to $\BH^{-\frac12}(\Div, \Gamma_R)$, and by the boundedness of the trace operator $\gamma_T$,  the operator $\calT_2(\kappa)\gamma_T$ is compact from $\BX$ into $\BH^{-\frac12}(\Div, \Gamma_R)$.
\end{proof}

Next we define two operators $A(\kappa): \BX\to \BX$ and $K(\kappa): \BX\to \BX$, $\kappa\in \Lambda$, respectively, by
\begin{align*}
(A(\kappa)\Bvarphi,\Bpsi)_{\BX} &=  (\curl\Bvarphi, \curl\Bpsi) - \kappa^2(\Bvarphi,\Bpsi) -\rmi\kappa \langle \calT_1(\kappa) \Bvarphi_T, \Bpsi_T\rangle_{\Gamma_R}\,\,\forall\Bvarphi, \Bpsi\in \BX,\\
(K(\kappa)\Bvarphi,\Bpsi)_{\BX} & = - \rmi\kappa \langle \calT_2(\kappa)\Bvarphi_T, \Bpsi_T\rangle_{\Gamma_R}\quad\forall\Bvarphi, \Bpsi\in \BX.
\end{align*}
It is clear that 
\begin{equation}
B(\kappa) = A(\kappa) + K(\kappa).
\label{eq:operator-split}
\end{equation}

\begin{theorem}\label{thm:Fredholm}
The operator function $B(\kappa): \BX\to \BX$ is holomorphic and Fredholm with index zero on $\Lambda$. Furthermore, the resolvent set $\rho(B)$ is non-empty.
\end{theorem}

\begin{proof}
Differentiating the coefficients in the second terms of \eqref{eq:Calderon}, one has that
\begin{equation*}
\left(\frac{z_n^{(1)}(\kappa R)}{\kappa h_n^{(1)}(\kappa R)}\right)' 
= -\frac{z_n^{(1)}(\kappa R)}{\kappa^2 h_n^{(1)}(\kappa R)}
+ \kappa R\frac{z_n^{(1)'}(\kappa R)}{\kappa^2 h_n^{(1)}(\kappa R)}
- \kappa R\frac{z_n^{(1)}(\kappa R)}{\kappa^2 h_n^{(1)}(\kappa R)} \frac{h_n^{(1)'}(\kappa R)}{h_n^{(1)}(\kappa R)}.
\end{equation*}
The definition of $z_n^{(1)}(z)$ and \eqref{eq:hprime} imply that 
\begin{align}
z_n^{(1)}(z) &= -n h_n^{(1)}(z) + z h_{n-1}^{(1)}(z),\nn\\
z_n^{(1)'}(z) &= -n h_n^{(1)'}(z) + h_{n-1}^{(1)}(z) + z h_{n-1}^{(1)'}(z) \nn\\
&= \frac{n(n+1)}{z} h_n^{(1)}(z) - (2n-1) h_{n-1}^{(1)}(z) + z h_{n-2}^{(1)}(z).
\label{eq:zprime}
\end{align}
By using \eqref{eq:hprime}, \eqref{eq:zprime} and \eqref{eq:hasymp}, we obtain
\begin{align}
z\frac{h_n^{(1)'}(z)}{h_n^{(1)}(z)} &= -(n+1) + z\frac{h_{n-1}^{(1)}(z)}{h_n^{(1)}(z)} = -(n+1) + \frac{z^2}{2n-1} + O\left(\frac{1}{n^2}\right),
\label{eq:hprimedh}\\
z \frac{z_n^{(1)'}(z)}{h_n^{(1)}(z)} &= n(n+1) -(2n-1) z \frac{h_{n-1}^{(1)}(z)}{h_n^{(1)}(z)} + z^2 \frac{h_{n-2}^{(1)}(z)}{h_n^{(1)}(z)}\nn\\
&= n(n+1) - z^2 + O\left(\frac{1}{n}\right).
\label{eq:zprimedh}
\end{align}
The estimates \eqref{eq:zdh}, \eqref{eq:hprimedh} and \eqref{eq:zprimedh} together yield
\begin{equation}
\left(\frac{z_n^{(1)}(\kappa R)}{\kappa h_n^{(1)}(\kappa R)}\right)' 
= \frac{n}{\kappa^2} +O\left(\frac1n\right).
\label{eq:est1}
\end{equation}

We proceed by taking the derivative of the coefficients in the first terms of the series expansion of $\calT(\kappa)$:
\begin{align*}
\left(\kappa\frac{h_n^{(1)}(\kappa R)}{z_n^{(1)}(\kappa R)}\right)' 
&= \frac{h_n^{(1)}(\kappa R)}{z_n^{(1)}(\kappa R)}
+ \kappa R\frac{h_n^{(1)'}(\kappa R)}{z_n^{(1)}(\kappa R)}
- \kappa R\frac{h_n^{(1)}(\kappa R)}{z_n^{(1)}(\kappa R)} \frac{z_n^{(1)'}(\kappa R)}{z_n^{(1)}(\kappa R)},\\
&= \frac{h_n^{(1)}(\kappa R)}{z_n^{(1)}(\kappa R)} \left(1 + \kappa R \frac{h_n^{(1)'}(\kappa R)}{h_n^{(1)}(\kappa R)} - \kappa R \frac{z_n^{(1)'}(\kappa R)}{h_n^{(1)}(\kappa R)} \frac{h_n^{(1)}(\kappa R)}{z_n^{(1)}(\kappa R)}\right).
\end{align*}
Combination of \eqref{eq:zdh}, \eqref{eq:hprimedh} and \eqref{eq:zprimedh} shows that
\begin{equation}
\left(\kappa\frac{h_n^{(1)}(\kappa R)}{z_n^{(1)}(\kappa R)}\right)'  = -\frac1n + O\left(\frac{1}{n^2}\right).
\label{eq:est2}
\end{equation}
From the estimates \eqref{eq:est1} and \eqref{eq:est2}, we obtain
\begin{align*}
\N{\calT'(\kappa)\Bxi}_{\BH^{-\frac12}(\Div,\Gamma_R)}^2 & \le CR^2\sum_{n=1}^{\infty}\sum_{m=-n}^n \sqrt{n(n+1)} \frac{\abs{a_n^m}^2}{n^2} +  \frac{1}{\sqrt{n(n+1)}} n^2 \abs{b_n^m}^2 \\
&\le C \N{\Bxi}_{\BH^{-\frac12}(\Curl, \Gamma_R)},
\end{align*}
which implies that $\calT'(\kappa): \BH^{-\frac12}(\Curl,\Gamma_R) \to \BH^{-\frac12}(\Div,\Gamma_R)$ is well-defined and bounded. Taking the derivative of \eqref{eq:Bk}, we have that
$B(\kappa)$ is holomorphic with $B'(\kappa): \BX\to \BX$ satisfying
\begin{equation}
(B'(\kappa)\Bvarphi, \Bpsi)_{\BX} = -2\kappa(\Bvarphi, \Bpsi) - \rmi\langle \calT(\kappa)\Bvarphi_T, \Bpsi_T\rangle_{\Gamma_R} - \rmi\kappa \langle \calT'(\kappa)\Bvarphi_T, \Bpsi_T\rangle_{\Gamma_R}.
\end{equation}
Due to \eqref{eq:Tbounded} and Lemma \ref{lem:inf-sup}, $A(\kappa)$ is bounded and invertible. 

We now prove the compactness of $K(\kappa)$. Indeed, $K(\kappa)= -\rmi\kappa M \calT_2(\kappa)\gamma_T$, where $M$ is the operator defined in \eqref{eq:M}. Using the boundedness of $M$ together with Lemma \ref{lem:compact}, we conclude that $K(\kappa)$ is compact. Consequently, $B(\kappa)$ is a Fredholm operator of index zero for every $\kappa \in \Lambda$.

Note that the operator $B(\kappa)$ is invertible for $\Im(\kappa)\ge 0$ and $B(\kappa)$ is holomorphic. Given any $\kappa_0\in\bbR$, there exists a neighbourhood $U(\kappa_0,\epsilon)$ of $\kappa_0$ in the complex plane, such that $B(\kappa)$ remains invertible for all $\kappa\in U(\kappa_0,\epsilon)$ with $0<\epsilon<\alpha_2$. Analogously, the operator $B(\kappa)$ is a Fredholm operator of index zero for $\kappa\in\calK$ with 
\begin{equation*}
\calK=\{\, \kappa=\kappa_1+\rmi\kappa_2\in\bbC\, |\, \abs{\kappa}\le\alpha_1, \, \kappa_2<-\epsilon/2<0 \, \} \cap \tilde\Lambda.
\end{equation*} 
Since $B(\kappa)$ is invertible on $\calK\cap U(\kappa_0,\epsilon)$, the resolvent $\rho(B)$ is nonempty in $\calK$. By the theory of holomorphic Fredholm operators, the specturm $\sigma(B)$ is discrete in $\calK$. Consequently, $\sigma(B)$ is also discrete in $\Lambda$ and $\rho(B)$ remains non-empty in $\Lambda$.
\end{proof}

\subsection{Truncation of the Calder\'on operator}

For numerical implementation, the infinite series \eqref{eq:Calderon} for $\calT(\kappa)$ must be truncated, resulting in a finite-term approximation defined by
\begin{equation}
\calT^N(\kappa)\Bxi = \sum_{n=1}^{N}\sum_{m=-n}^n 
\frac{\rmi\kappa R \,h_n^{(1)}(\kappa R)}{z_n^{(1)}(\kappa R)}a_n^m\BU_n^m
+ \frac{z_n^{(1)}(\kappa R)}{\rmi\kappa R\,h_n^{(1)}(\kappa R)}b_n^m\BV_n^m.
\label{eq:CalderonN}
\end{equation}
The integer $N$ is called the truncation order. Similar to \eqref{eq:Tbounded}, the operator $\calT^N(\kappa)$ is also bounded, i.e.
\begin{equation}
\N{\calT^N(\kappa)\Bxi}_{\BH^{-\frac12}(\Div,\Gamma_R)} \le C \N{\Bxi}_{\BH^{-\frac12}(\Curl, \Gamma_R)}.
\label{eq:TNbound}
\end{equation}

\begin{lemma}\label{lem:TruncationError}
For $\calT(\kappa)$ and  $\calT^N(\kappa)$ defined by \eqref{eq:Calderon} and \eqref{eq:CalderonN}, respectively, it holds that
\begin{equation}
\N{(\calT(\kappa)-\calT^N(\kappa))\Bxi}_{\BH^{-\frac12}(\Div,\Gamma_R)} \le C\frac{\gamma(N,\Bxi)}{N^s} \N{\Bxi}_{\BH^{s-\frac12}(\Curl, \Gamma_R)}
\end{equation}
for all $\Bxi\in\BH^{-\frac12}(\Curl, \Gamma_R)$, where $C>0$ is a constant depending on $\kappa R$ but not on $\Bxi$ and $N$. Furthermore,
\begin{equation*}
\gamma_s(N,\Bxi) = \frac{\sum_{n=N+1}^{\infty}\sum_{m=-n}^n (n(n+1))^{s-\frac12} \abs{a_n^m}^2 +  (n(n+1))^{s+\frac12} \abs{b_n^m}^2}
{\sum_{n=1}^{\infty}\sum_{m=-n}^n (n(n+1))^{s-\frac12} \abs{a_n^m}^2 +  (n(n+1))^{s+\frac12} \abs{b_n^m}^2}  \in[0, 1]
\end{equation*}
satisfying $\gamma_s(N,\Bxi) \to 0$ as $N\to\infty$.
\end{lemma}

\begin{proof}
Subtracting equation \eqref{eq:CalderonN} from \eqref{eq:Calderon}, we have that
\begin{equation*}
(\calT(\kappa)-\calT^N(\kappa))\Bxi = \sum_{n=N+1}^{\infty}\sum_{m=-n}^n 
\frac{\rmi\kappa R \,h_n^{(1)}(\kappa R)}{z_n^{(1)}(\kappa R)}a_n^m\BU_n^m
+ \frac{z_n^{(1)}(\kappa R)}{\rmi\kappa R\,h_n^{(1)}(\kappa R)}b_n^m\BV_n^m.
\end{equation*}
By the definitions of the norms on the spaces $\BH^{-\frac12}(\Div, \Gamma_R)$ and $\BH^{-\frac12}(\Curl, \Gamma_R)$, using \eqref{eq:zdh}, we obtain that
\begin{align*}
&\N{(\calT(\kappa)-\calT^N(\kappa))\Bxi}_{\BH^{-\frac12}(\Div,\Gamma_R)}^2 \\
\le& CR^2\sum_{n=N+1}^{\infty}\sum_{m=-n}^n \sqrt{n(n+1)} \frac{\abs{a_n^m}^2}{n^2} +  \frac{1}{\sqrt{n(n+1)}} n^2 \abs{b_n^m}^2 \\
\le& CR^2\sum_{n=N+1}^{\infty}\sum_{m=-n}^n \frac{1}{\sqrt{n(n+1)}} \abs{a_n^m}^2 +  \sqrt{n(n+1)} \abs{b_n^m}^2 \\
\le& \frac{CR^2}{N^{s}}\sum_{n=N+1}^{\infty}\sum_{m=-n}^n (n(n+1))^{s-\frac12} \abs{a_n^m}^2 +  (n(n+1))^{s+\frac12} \abs{b_n^m}^2 \\
\le& C\frac{\gamma_s(N,\Bxi)}{N^{s}} \N{\Bxi}_{\BH^{s-\frac12}(\Curl, \Gamma_R)}.
\end{align*}
This completes the proof.
\end{proof}

Using the truncated Calder\'on operator $\calT^N(\kappa)$, the approximation $B^N(\kappa)$ of $B(\kappa)$ is defined by
\begin{equation}
(B^N(\kappa)\Bvarphi, \Bpsi)_{\BX} = (\curl\Bvarphi, \curl\Bpsi) - \kappa^2(\Bvarphi, \Bpsi) - \rmi\kappa\langle\calT^N(\kappa)\Bvarphi_T, \Bpsi_T\rangle_{\Gamma_R}
\label{eq:BN}
\end{equation}
for all $\Bvarphi, \Bpsi\in \BX$. The truncated eigenvalue problem for $B^N(\kappa)$ is to find $(\kappa^N, \BE^N)\in \Lambda\times \BX$ such that 
\begin{equation}
B^N(\kappa^N)\BE^N = 0\quad\text{in}\ \ \BX.
\label{eq:eig-XN}
\end{equation}
The operator $B^N(\kappa)$ can be split analogously
\begin{equation}
B^N(\kappa) = A^N(\kappa) + K^N(\kappa),
\end{equation}
where $A^N(\kappa): \BX\to \BX$ and $K^N(\kappa): \BX\to \BX$ are defined, respectively, by
\begin{align*}
(A^N(\kappa)\Bvarphi,\Bpsi)_{\BX} &=  (\curl\Bvarphi, \curl\Bpsi) - \kappa^2 (\Bvarphi,\Bpsi) -\rmi\kappa \langle \calT^N_1(\kappa) \Bvarphi_T, \Bpsi_T\rangle_{\Gamma_R}\,\forall\Bvarphi, \Bpsi\in \BX,\\
(K^N(\kappa)\Bvarphi,\Bpsi)_{\BX} & = - \rmi\kappa \langle \calT^N_2(\kappa)\Bvarphi_T, \Bpsi_T\rangle_{\Gamma_R}\,\forall\Bvarphi, \Bpsi\in \BX.
\end{align*}
Here $\calT^N_1(\kappa)$ and $\calT^N_2(\kappa)$ denote the truncated counterparts of $\calT_1(\kappa)$ and $\calT_2(\kappa)$, respectively.

\section{Edge Element Method}\label{Sec3}

In this section, we present and analyze the edge element discretization for the eigenvalue problem \eqref{eq:eig-XN}. Let $\calM_n := \calM_{h_n}$ be a sequence of regular tetrahedral meshes of $\Omega$ with mesh size $h_n \to 0$ as $n \to \infty$. To accommodate the boundaries $\Gamma_D$ and $\Gamma_R$, each element $K \in \calM_n$ is permitted to have one curved edge or one curved face, yielding the decomposition $\Omega = \cup_{K \in \calM_n} K$. 
For each curved element $K$, we extend the standard mapping $F_K: \hat K\to K$ using the approach introduced by Dubois (\cite{Dubios1990SINUM,Monk03}), which maps the relevant edge or face of the reference element $\hat K$ exactly onto the curved boundary.

Let $R_1(\hat K)$ denote the lowest-order N\'ed\'elec edge element space on the reference element $\hat K$. The space on the element $K$ is then defined via the mapping $F_K$ by
\begin{equation*}
R_1(K) = \{ \,\Bvarphi: \,\Bvarphi\circ F_K = (\rmd F_K)^{-T} \hat\Bvarphi\ \text{for some} \ \hat\Bvarphi\in R_1(\hat K)\, \},
\end{equation*}
where $\rmd F_K$ denotes the Jacobi matrix of $F_K$.
The finite element discretization is based on the space $\BX_n \subset \BX$ on $\calM_n$ defined by
\begin{equation*}
\BX_n = \{ \Bvarphi\in\BX: \,\Bvarphi|_K \in R_1(K), \ \  K\in\calM_n \}.
\end{equation*}
Define the operator $B^N_n(\kappa): \BX_n\to \BX_n$ such that
\begin{equation}
(B^N_n(\kappa)\Bvarphi_n, \Bpsi_n)_{\BX} = (\curl\Bvarphi_n, \curl\Bpsi_n) - \kappa^2(\Bvarphi_n, \Bpsi_n) - \rmi\kappa\langle\calT^N(\kappa)\Bvarphi_{nT}, \Bpsi_{nT}\rangle_{\Gamma_R}
\label{eq:BNn}
\end{equation}
for all $\Bvarphi_n,\Bpsi_n\in \BX_n$.


The finite element approximation to \eqref{eq:eig-XN} is to find $(\kappa^N_n, \BE^N_n)\in \Lambda\times \BX_n$, such that 
\begin{equation}
B^N_n(\kappa^N_n)\BE^N_n = 0\quad\text{in}\ \ \BX_n.
\label{eq:eig-XNn}
\end{equation}
In a similar way, we define $A^N_n(\kappa): \BX_n\to \BX_n$ and $K^N_n(\kappa): \BX_n\to \BX_n$, respectively, as the finite element approximation of $A^N(\kappa)$ and $K^N(\kappa)$, and thus
\begin{equation}
B^N_n(\kappa) = A^N_n(\kappa) + K^N_n(\kappa).
\end{equation}

The rest of this section is devoted to the convergence of eigenvalues of \eqref{eq:eig-XNn} by employing the abstract approximation theory for holomorphic Fredholm operator functions based on the notion of regular convergence. We first introduce the definitions of operator convergence and outline the approximation theory (see~\cite{Karma1,Karma2,Vnk76}). Then we verify that the finite element approximation is indeed regular, from which the convergence of eigenvalues follows.

\subsection{Abstract approximation theory}

Let $\mathcal{X}, \mathcal{Y}$ be Banach spaces, and ${\mathcal{X}_n}, {\mathcal{Y}_n}, n\in\bbN$ be their respective approximating Banach spaces. Let $\calP=\{p_n\}_{n\in\bbN}$ and $\calQ=\{q_n\}_{n\in\bbN}$ be two operator sequences with $p_n\in\calL(\calX,\calX_n)$ and $q_n\in\calL(\calY,\calY_n)$ such that
\begin{equation*}
\N{p_nx}_{\calX_n}\to\N{x}_{\calX},\ \N{q_ny}_{\calY_n}\to\N{y}_{\calY},\quad n\to\infty,
\end{equation*}
and 
\begin{align*}
&\N{p_n(\alpha x + \alpha' x') - (\alpha p_nx + \alpha' p_nx')} \to 0, \ \forall\alpha,\alpha'\in\bbC,\ \ n\to\infty,\\
&\N{q_n(\alpha y + \alpha' y') - (\alpha q_ny + \alpha' q_ny')} \to 0, \ \forall\alpha,\alpha'\in\bbC,\ \ n\to\infty.
\end{align*}
for all $x,x'\in\calX, \ y,y'\in\calY$.

\begin{definition}
A sequence $\{x_n\}_{n\in\bbN}$ with $x_n\in\calX_n$ is $\calP$-compact if, for each subsequence $\{x_n\}_{n\in\bbN'}, N'\subset N$, there exists $N''\subset N'$ and $x\in\calX$ such that $\N{x_n-p_nx}_{\calX_n}\to 0, n\to\infty, n\in N''$. A sequence $\{y_n\}_{n\in\bbN}$ with $y_n\in\calY_n$ being $\calQ$-compact is defined similarly.
\end{definition}

\begin{definition}
Let $F\in\calL(\calX,\calY)$ and $F_n\in\calL(\calX_n,\calY_n)$, $n \in \mathbb N$.
\begin{itemize}
\item $F_n$ converges to $F$, if $\N{F_np_nx - q_nFx}_{\calY_n}\to 0, \ n\to\infty,\ \forall x\in\calX$.
\item $F_n$ converges uniformly to $F$, if $\N{F_np_n - q_nF}_{\calY_n}\to 0, \ n\to\infty$.
\item $F_n$ converges stably to $F$, if $F_n$ converges to $F$ and there exists $n_0$ such that $F_n$ is invertible with $\N{F_n^{-1}}\le C, n\ge n_0$.
\item $F_n$ converges regularly to $F$, if $F_n$ converges to $F$ and $\{F_nx_n\}$ is $\calQ$-compact implies that $\{x_n\}$ is $\calP$-compact.
\end{itemize}
\end{definition}

Denote by $\Phi_0(\Lambda,\calL(\calX,\calY))$ the set of all holomorphic operator functions $F(\cdot)$'s such that $F(\kappa)\in\calL(\calX,\calY)$ is a Fredholm operator with index zero for each $\kappa\in\Lambda$. The eigenvalue problem for $F(\cdot)\in\Phi_0(\Lambda,\calL(\calX,\calY))$ is to find $(\lambda,x)\in\Lambda\times\calX$, such that
\begin{equation*}
F(\lambda)x = 0 \ \ \text{in}\ \ \calY.
\end{equation*}

\begin{definition}
An ordered sequence of elements $x_0, x_1, \cdots, x_j$ in $\calX$ is called a Jordan chain of $F$ at an eigenvalue $\lambda$ if
\begin{equation*}
F(\lambda)x_j + \frac{1}{1!}F^{(1)}(\lambda)x_{j-1} + \cdots \frac{1}{j!}F^{(j)}(\lambda)x_0 = 0,\quad j=0,1,\cdots,
\end{equation*}
where $F^{(j)}$ denotes the $j$-th derivative. Elements of any Jordan chain at an eigenvalue $\lambda$ are called generalized eigenelements of $F$.
The closed linear hull of all generalized eigenelements, denoted by $G(F,\lambda)$, is called the generalized eigenspace of $F$ at $\lambda$.
\end{definition}

If $F(\cdot)\in\Phi_0(\Lambda,\calL(\calX,\calY))$ and $\rho(F)\neq\emptyset$, then the spectrum $\sigma(F)$ has no cluster points in $\Lambda$. Every $\lambda\in\sigma(F)$ is an eigenvalue with finite dimensional generalized eigenspace $G(F,\lambda)$. In the following, we consider a sequence of operator functions $F_n(\cdot)\in\Phi_0(\Lambda,\calL(\calX_n,\calY_n)), n\in\bbN$.

\begin{theorem}\label{thm:Karma1}(\cite[Theorem 1]{Karma2})
Let $F(\cdot)\in\Phi_0(\Lambda,\calL(\calX,\calY))$ with $\rho(F)\neq\emptyset$. 
Let $F_n(\cdot)\in\Phi_0(\Lambda,\calL(\calX_n,\calY_n)), n\in\bbN$ be a sequence of operator functions.
Assume that $\{F_n(\cdot)\}_{n\in\bbN}$ is equibounded on any compact $\Lambda_0\subset\Lambda$, i.e.,
\begin{equation*}
\N{F_n(\lambda)} \le C,\quad \forall\lambda\in\Lambda_0, \ n\in\bbN,
\end{equation*}
and $F_n(\lambda)$ converges regularly to $F(\lambda)$ for every $\lambda\in\Lambda$.
Denote by $\nu(F,\Lambda_0)$ the sum of the algebraic multiplicities
of all the eigenvalues of $F(\cdot)$ in $\Lambda_0$.
Then for every compact $\Lambda_0\subset \Lambda$ with
$\partial\Lambda_0\subset\rho(F)$ and sufficiently large $n$,
\begin{equation*}
\nu(F_n,\Lambda_0) = \nu(F,\Lambda_0).
\end{equation*}
\end{theorem}

\begin{theorem}\label{thm:Karma2}(\cite[Theorem 2]{Karma2})
Let the assumptions of Theorem \ref{thm:Karma1} hold.
Let $\Lambda_0\subset \Lambda$ be a compact set such that
$\partial\Lambda_0\subset\rho(F)$ and $\sigma(F)\cap\Lambda_0=\{\lambda_0\}$.
Then for sufficiently large $n$ and each $\lambda_n\in\sigma(F_n)\cap\Lambda_0$ it holds
\begin{equation}
\abs{\lambda_n-\lambda_0} \le C \varepsilon_n^{1/r},
\label{eq:en}
\end{equation}
where
\begin{equation*}
\varepsilon_n = \sup_{\lambda\in\partial\Lambda_0, x\in G(F,\lambda_0), \N{x}_{\calX}=1} \N{F_n(\lambda)p_nx - q_nF(\lambda)x}_{\calY_n}.
\end{equation*}
Here $r$ is the maximum length of all Jordan chains of $\lambda$.
\end{theorem}

Define the adjoint eigenvalue problem
\begin{equation*}
F^*(\lambda)x := [F(\overline{\lambda})]^*x =0 \ \ \text{in}\ \ \calY.
\end{equation*}
It is easy to see that $F^*(\lambda)$ is holomorphic with respect to $\lambda$. The generalized eigenspace of $F^*$ at $\overline{\lambda}$ is denoted by $G(F^*,\overline{\lambda})$.

\begin{theorem}\label{thm:Karma3}(\cite[Theorem 3]{Karma2})
Let the assumptions of Theorem \ref{thm:Karma1} hold.
Let $\Lambda_0\subset \Lambda$ be a compact set such that
$\partial\Lambda_0\subset\rho(F)$ and $\sigma(F)\cap\Lambda_0=\{\lambda_0\}$.
Let $r_n\in \calL(\calX_n,\calX)$ and $q_n\in \calL(\calY,\calY_n)$ 
be such that
\begin{enumerate}
\item
$\Vert r_n\Vert\leq C$, $\forall n$,
\item
$\Vert r_n p_n x - p_n x\Vert_{\calX_n}\rightarrow 0$, $n\rightarrow \infty$, $\forall x\in G(F,\lambda_0)$,
\end{enumerate}
and let $F_n(\lambda) = q_nF(\lambda)r_n$, $\forall \lambda\in \Lambda$. 
Then for sufficiently large $n$, \eqref{eq:en} holds
with $\epsilon_n = d_n d_n^*$, where
\begin{align*}
d_n &= \max_{x\in G(F,\lambda_0), \Vert x\Vert_{\calX} = 1}{\rm{dist}}(x,r_n\calX_n),
\\
d_n^* &= \max_{y\in G(F^*,\overline{\lambda}_0), \Vert y\Vert_{\calY} = 1}{\rm{dist}}(y,q_n^*\calY_n).
\end{align*}
\end{theorem}

\subsection{Convergence result for the eigenvalues}
 
Assume that $\calX_n\subset\calX, \calY_n\subset\calY$, $p_n,q_n$ are projections such that
\begin{equation*}
\N{x-p_nx}_{\calX}\to 0,\ \N{y-q_ny}_{\calY}\to 0,\quad  n\to\infty
\end{equation*}
for all $x\in\calX, y\in\calY$. 
Let $A,K\in\calL(\calX,\calY)$ and $A_n,K_n\in\calL(\calX_n,\calY_n)$, the following lemma gives two sufficient conditions for the regular convergence of $A_n+K_n$ to $A+K$.

\begin{lemma}\label{lem:RegularConvergence}(\cite[Lemma 3.4, 3.5]{XiGongSun2024})
Assume that $A$ is invertible, $K$ is compact, $A_n$ converges stably to $A$. If $K_n$ converges uniformly to $K$, or $K_n=q_nK|_{\calX_n}$,
then $A_n+K_n$ converges regularly to $A+K$.
\end{lemma}
Now we prove some  properties of the related operators.

\begin{lemma}\label{lem:Fredholm-BNn}
$B^N(\kappa)$ and $B_n^N(\kappa)$ are holomorphic Fredholm operator functions of index zero on $\Lambda$.
\end{lemma}

\begin{proof}
As in the proof of Theorem \ref{thm:Fredholm}, one can show that $B^N(\kappa)$ and $B_n^N(\kappa)$ are holomorphic operator functions, and $B^N(\kappa)$ is a Fredholm operator with index zero on $\Lambda$. Since $\BX_n$ is finite dimensional, $B_n^N(\kappa)$ is also a Fredholm operator with index zero.
\end{proof}

\begin{lemma}\label{lem:equibounded}
On each compact $\Lambda_0\subset\Lambda$, $B^N(\kappa)$ and $B_n^N(\kappa)$ are equibounded.
\end{lemma}

\begin{proof}
The equiboundedness of $B^N(\kappa)$ and $B_n^N(\kappa)$ follows directly from \eqref{eq:TNbound} and the definitions \eqref{eq:BN} and \eqref{eq:BNn}.
\end{proof}

\begin{lemma}\label{lem:invertible}
$A^N(\kappa)$ and $A_n^N(\kappa)$ are invertible with
\begin{equation*}
\N{(A^N(\kappa))^{-1}}\le C,\quad \N{(A_n^N(\kappa))^{-1}}\le C,\quad\forall N\in\bbN, \,n\in\bbN.
\end{equation*}
\end{lemma}

\begin{proof}
The proof can be obtained analogously to that of Theorem \ref{thm:Fredholm}.
\end{proof}

Let $\calX = \calY = \BX,\ \calX^N = \calY^N = \BX$, and the projections $p^N$ and $q^N$ be the identity operator from $\BX$ to itself.

\begin{lemma}\label{lem:BN-regular}
For $\kappa\in \Lambda$, $B^N(\kappa)$ converges regularly to $B(\kappa)$. Moreover, for sufficiently large $N$, $\rho(B^N)\neq\emptyset$.
\end{lemma}

\begin{proof}
For $\Bvarphi\in \BX$, $\Bpsi\in \BX$, let $a_{1n}^m,\,b_{1n}^m$ and $a_{2n}^m,\,b_{2n}^m$ be the Fourier coefficients for $\Bvarphi_T$ and $\Bpsi_T$, respectively. The definition of $A(\kappa) $ and $A^N(\kappa) $ implies that
\begin{align*}
&(q^N A(\kappa) \Bvarphi - A^N(\kappa) p^N \Bvarphi, \Bpsi)_{\BX} =( (A(\kappa) - A^N(\kappa))\Bvarphi, \Bpsi)_{\BX} \\
=&-\rmi\kappa \langle (\calT_1(\kappa) -\calT_1^N(\kappa) )\Bvarphi_T,\Bpsi_T \rangle_{\Gamma_R}\\
=& -\rmi\kappa R^2 \sum_{n=N+1}^{\infty}\sum_{m=-n}^n \frac{\rmi\kappa R}{\delta_n(\rmi)}a_{1,n}^m\overline{a_{2,n}^m} +  \frac{\delta_n(\rmi)}{\rmi\kappa R}b_{1,n}^m\overline{b_{2,n}^m}.
\end{align*}
By using \eqref{eq:zdh}, we obtain
\begin{align*}
&(q^N A(\kappa) \Bvarphi - A^N(\kappa) p^N \Bvarphi, \Bpsi)_{\BX} \\
\le& CR^2 \sum_{n=N+1}^{\infty}\sum_{m=-n}^n (n(n+1))^{-\frac12}\abs{a_{1,n}^m}\abs{a_{2,n}^m} + (n(n+1))^{\frac12}\abs{b_{1,n}^m}\abs{b_{2,n}^m}\\
\le& C\gamma_0(N,\Bvarphi_T) \N{\Bvarphi_T}_{\BH^{-\frac12}(\Curl, \Gamma_R)} \N{\Bpsi_T}_{\BH^{-\frac12}(\Curl, \Gamma_R)}\\
\le&  C\gamma_0(N,\Bvarphi_T) \N{\Bvarphi}_{\BX} \N{\Bpsi}_{\BX}.
\end{align*}
Therefore,
\begin{equation*}
\N{q^N A(\kappa)  \Bvarphi - A^N(\kappa)  p^N \Bvarphi}_{\BX} \le C\gamma_0(N,\Bvarphi_T) \N{\Bvarphi}_{\BX} \to 0,\quad N\to\infty,\ \forall\Bvarphi\in \BX,
\end{equation*}
together with the fact that $A^N(\kappa) $ is invertible by Lemma \ref{lem:invertible}, $A^N(\kappa) $ converges stably to $A(\kappa)$. 
On the other hand, 
\begin{align*}
&(q^N K(\kappa) \Bvarphi - K^N(\kappa) p^N \Bvarphi, \Bpsi)_{\BX} \\
=&( (K(\kappa) - K^N(\kappa))\Bvarphi, \Bpsi)_{\BX}
= -\rmi\kappa \langle (\calT_2(\kappa)-\calT_2^N(\kappa))\Bvarphi_T,\Bpsi_T \rangle_{\Gamma_R}\\
=& -\rmi\kappa R^2 \sum_{n=N+1}^{\infty}\sum_{m=-n}^n 
\rmi\kappa R \left( \frac{1}{\delta_n(\kappa)} - \frac{1}{\delta_n(\rmi)} \right) a_{1,n}^m\overline{a_{2,n}^m}
+ \frac{1}{\rmi\kappa R}(\delta_n(\kappa) - \delta_n(\rmi)) b_{1,n}^m\overline{b_{2,n}^m}.
\end{align*}
From \eqref{eq:zdh} and \eqref{eq:zdh-diff}, we have
\begin{align*}
&(q^N K(\kappa) \Bvarphi - K^N(\kappa) p^N \Bvarphi, \Bpsi)_{\BX}\\
\le& CR^2 \sum_{n=N+1}^{\infty}\sum_{m=-n}^n (n(n+1))^{-\frac32}\abs{a_{1,n}^m}\abs{a_{2,n}^m} + (n(n+1))^{-\frac12}\abs{b_{1,n}^m}\abs{b_{2,n}^m}\\
\le& \frac{CR^2}{N} \sum_{n=N+1}^{\infty}\sum_{m=-n}^n (n(n+1))^{-\frac12}\abs{a_{1,n}^m}\abs{a_{2,n}^m} + (n(n+1))^{\frac12}\abs{b_{1,n}^m}\abs{b_{2,n}^m}\\
\le& \frac{C}{N} \N{\Bvarphi_T}_{\BH^{-\frac12}(\Curl, \Gamma_R)} \N{\Bpsi_T}_{\BH^{-\frac12}(\Curl, \Gamma_R)}
\le \frac{C}{N} \N{\Bvarphi}_{\BX} \N{\Bpsi}_{\BX}.
\end{align*}
Therefore,
\begin{equation*}
\N{q^N K(\kappa) \Bvarphi - K^N(\kappa) p^N \Bvarphi}_{\BX} \le \frac{C}{N} \N{\Bvarphi}_{\BX} \to 0,\quad N\to\infty,\ \forall\Bvarphi\in \BX,
\end{equation*}
and $K^N(\kappa)$ converges uniformly to $K(\kappa)$. Since $A(\kappa)$ is invertible and $K(\kappa)$ is compact, by Lemma \ref{lem:RegularConvergence}, $B^N(\kappa)$ converges regularly to $B(\kappa)$.  Together with Lemma \ref{lem:equibounded}, for sufficient large $N$, we conclude $\nu(B^N,\Lambda_0) = \nu(B,\Lambda_0)$ from Theorem \ref{thm:Karma1} for every compact $\Lambda_0\subset \Lambda$ with $\partial\Lambda_0\subset\rho(B)$. This implies that $\sigma(B^N)\cap\Lambda_0$ is discrete and thus $\rho(B^N)\neq\emptyset$.
\end{proof}

Let $\calX = \calY = \BX, \ \calX_n = \calY_n = \BX_n$. Let $q_n: \BX\to\BX_n$ be the projection
\begin{equation}
(q_n\Bvarphi, \Bpsi_n)_{\BX} =  (\Bvarphi, \Bpsi_n)_{\BX} \quad\forall\Bpsi_n\in\BX_n,
\label{eq:qn}
\end{equation}
and $p_n = q_n$.

\begin{lemma}\label{lem:BNn-regular}
For $\kappa\in \Lambda$ and $N$ sufficient large, $B^N_n(\kappa)$ converges regularly to $B^N(\kappa)$.
\end{lemma}

\begin{proof}
Based on the definitions of the projections $q_n$ and $p_n$, we have
\begin{equation*}
\N{\Bvarphi-q_n\Bvarphi}_{\BX}\to 0,\quad \N{\Bvarphi-p_n\Bvarphi}_{\BX}\to 0,\qquad n\to\infty,\ \ \forall \Bvarphi\in\BX.
\end{equation*}
For any $\Bvarphi_n, \Bpsi_n\in\BX_n$, 
\begin{equation*}
(q_nK^N(\kappa)\Bvarphi_n,\Bpsi_n)_{\BX} = (K^N(\kappa)\Bvarphi_n,\Bpsi_n)_{\BX} = (K_n^N(\kappa)\Bvarphi_n,\Bpsi_n)_{\BX}
\end{equation*}
implies that $q_nK^N(\kappa)|_{\BX_n} = K_n^N(\kappa)$.
Analogously, $q_nA^N(\kappa)|_{\BX_n} = A_n^N(\kappa)$. Thus
\begin{equation*}
\N{q_nA^N(\kappa)\Bvarphi - A^N_n(\kappa)p_n\Bvarphi}_{\BX}
= \N{q_nA^N(\Bvarphi-p_n\Bvarphi)}_{\BX}
\le C\N{\Bvarphi-p_n\Bvarphi}_{\BX}\rightarrow 0.
\end{equation*}
Since $A^N_n(\kappa)$ is invertible by Lemma \ref{lem:invertible}, the above estimate shows that $A^N_n(\kappa)$ converges stably to $A^N(\kappa)$. 

In addition, by the same lemma, $A^N(\kappa)$ is invertible.
Using the same approach as in the proofs of Lemma \ref{lem:compact} and Theorem \ref{thm:Fredholm}, it follows that both $\calT^N_2(\kappa)\gamma_T$ and $K^N(\kappa)$ are compact. Then Lemma \ref{lem:RegularConvergence} implies that $B^N_n(\kappa)$ converges regularly to $B^N(\kappa)$.
\end{proof}

Before analyzing the discretization error, we consider the computation of the dual product $\langle\calT(\kappa)\Bxi,\Beta\rangle_{\Gamma_R}$ for arbitrary tangential vectors $\Bxi$ and $\Beta$. Let $a_{n}^m(\Bxi),\,b_{n}^m(\Bxi)$ be the Fourier coefficients for $\Bxi$, i.e.,
\begin{align*}
&a_{n}^m(\Bxi) = \int_{\partial B_1}\Bxi(R,\hBx)\cdot\overline{\BU_n^m}(\hBx)\rmd\hBx, 
&&b_{n}^m(\Bxi) = \int_{\partial B_1}\Bxi(R,\hBx)\cdot\overline{\BV_n^m}(\hBx)\rmd\hBx.
\end{align*}
By using the Parseval's theorem, the dual product can be computed by 
\begin{equation*}
\langle\calT(\kappa)\Bxi, \Beta\rangle_{\Gamma_R} =  R^2 \sum_{n=1}^{\infty}\sum_{m=-n}^n \left(
\frac{\rmi\kappa R}{\delta_n(\kappa)}a_{n}^m(\Bxi)\overline{a_{n}^m(\Beta)}
+ \frac{\delta_n(\kappa)}{\rmi\kappa R}b_{n}^m(\Bxi)\overline{b_{n}^m(\Beta)}\right).
\end{equation*} 
For the truncated Calder\'on operator $\calT^N(\kappa)$, the dual product $\langle\calT^N(\kappa)\Bxi, \Beta\rangle_{\Gamma_R}$ follows analogously, obtained by summing the relation over $n=1$ to $N$. By the definition of $\BU_n^m$ and $\BV_n^m$, it is easy to see that $\overline{\BU_n^m}=\BU_n^{-m}$ and $\overline{\BV_n^m}=\BV_n^{-m}$ and thus
$a_{n}^m(\overline{\Bxi})=\overline{a_{n}^{-m}(\Bxi)}$ and $b_{n}^m(\overline{\Bxi})=\overline{b_{n}^{-m}(\Bxi)}$. Consequently, we have
\begin{align}
\langle\calT^N(\kappa)\Bxi,\Beta\rangle_{\Gamma_R}
&= R^2\sum_{n=1}^N\sum_{m=-n}^n\frac{i\kappa R}{\delta_n(\kappa)}a_n^m(\Bxi)\overline{a_n^m(\Beta)}
+ \frac{\delta_n(\kappa)}{i\kappa R}b_n^m(\Bxi)\overline{b_n^m(\Beta)}\nn\\
&= R^2\sum_{n=1}^N\sum_{m=-n}^n\frac{i\kappa R}{\delta_n(\kappa)}a_n^{-m}(\Bxi)\overline{a_n^{-m}(\Beta)}
+ \frac{\delta_n(\kappa)}{i\kappa R}b_n^{-m}(\Bxi)\overline{b_n^{-m}(\Beta)}\nn\\
&= R^2\sum_{n=1}^N\sum_{m=-n}^n\frac{i\kappa R}{\delta_n(\kappa)}\overline{a_n^m(\overline{\Bxi})}a_n^m(\overline{\Beta})
+ \frac{\delta_n(\kappa)}{i\kappa R}\overline{b_n^m(\overline{\Bxi})}b_n^m(\overline{\Beta})\nn\\
&= \langle\calT^N(\kappa)\overline{\Beta},\overline{\Bxi}\rangle_{\Gamma_R}.
\label{eq:T-adjoint}
\end{align}

Define the space 
\begin{equation*}
\BH^{t}(\curl,\Omega) = \{ \Bpsi\in\BH^t(\Omega): \,\curl\Bpsi\in\BH^t(\Omega) \}
\end{equation*}
for $1/2<t\le 1$.

\begin{theorem}\label{thm:convergence}
Let $\Lambda_0$ be a compact subset of $\Lambda$ such that $\partial\Lambda_0\subset\rho(B)$ and $\sigma(B)\cap\Lambda_0=\{\kappa_0\}$.
Assume that $\gamma_T G(B,\kappa_0)\subset \BH^{-\frac{1}{2}+s}(\Curl,\Gamma_R)$ and $G(B^N,\kappa^N)\subset \BH^{t}(\curl,\Omega)$ with $s>0$ and $1/2<t\le 1$.
Then for sufficiently large $N$ and $n$, $\nu(B,\kappa_0) = \nu(B^N,\kappa^N) = \nu(B^N_n,\kappa^N_n)$.
Moreover, for each $\kappa^N\in\sigma(B^N)\cap\Lambda_0$ and each $\kappa_n^N\in\sigma(B_n^N)\cap\Lambda_0$, it holds
\begin{align*}
&\abs{\kappa^N-\kappa_0}\le C_1 N^{-s/r},\\
&\abs{\kappa_n^N-\kappa^N}\le C_2 h_n^{2t/r^N},
\end{align*}
where $r$ and $r^N$ are the maximum length of Jordan chain of $\kappa$ and $\kappa^N$.
The constant $C_1$ depends on $\Lambda_0$ but is independent on $N$,  and $C_2$ depends on $\Lambda_0$ and $N$ but is independent on $h_n$. 
\end{theorem}

\begin{proof}
By definition,
\begin{align*}
&\left|(B^N(\kappa)p^N\Bvarphi - q^N B(\kappa)\Bvarphi,\Bpsi)_{\BX}\right|
= \left|-i\kappa\langle (\calT(\kappa) - \calT^N(\kappa))\Bvarphi_T,\Bpsi_T\rangle_{\Gamma_R}\right| 
\\
\leq\;& C\Vert (\calT(\kappa) - \calT^N(\kappa))\Bvarphi\Vert_{H^{-\frac{1}{2}}(\text{Div},\Gamma_R)}\Vert \Bpsi\Vert_{\BX}.
\end{align*}
Applying Theorem \ref{thm:Karma2},
we have
\begin{equation*}
|\kappa^N - \kappa_0|\leq \sup_{\kappa\in \partial\Lambda_0, \Bvarphi\in G(B,\kappa_0),
\Vert \Bvarphi\Vert_X = 1} C\Vert (\calT(\kappa) - \calT^N(\kappa))\Bvarphi\Vert_{H^{-\frac{1}{2}}(\text{Div},\Gamma_R)}^{1/r}.
\end{equation*}
Then by Lemma \ref{lem:TruncationError},
\begin{equation*}
|\kappa^N - \kappa_0|\leq C_1 N^{-s/r}.
\end{equation*}
Let $r_n$ be the embedding from $\BX_n$ to $\BX$. Then by the definition of $q_n$ in \eqref{eq:qn} and that of $B^N_n$ in \eqref{eq:BNn}, we have $r_n = q_n^*$.
Moreover, $B^N_n(\kappa) = q_n B^N(\kappa)r_n$ is a projection method. 
Using the property of $\calT^N(\kappa)$ in \eqref{eq:T-adjoint}, it holds that
\begin{align*}
&(B^N(\kappa)\Bvarphi,\Bpsi)_{\BX} = (\curl\Bvarphi, \curl\Bpsi) - \kappa^2(\Bvarphi,\Bpsi) - i\kappa\langle\calT^N(\kappa)\Bvarphi,\Bpsi\rangle_{\Gamma_R}\\
=\; & (\curl\overline{\Bpsi},\curl\overline{\Bvarphi}) - \kappa^2(\overline{\Bpsi},\overline{\Bvarphi}) - i\kappa\langle\calT^N(\kappa)\overline{\Bpsi},\overline{\Bvarphi}\rangle_{\Gamma_R} = (B^N(\kappa)\overline{\Bpsi},\overline{\Bvarphi})_{\BX}.
\end{align*}
Therefore $(B^N)^*({\kappa})\Bpsi := [B^N(\overline{\kappa})]^*\Bpsi = \overline{B^N(\overline{\kappa})\overline{\Bpsi}}$.
Then, successively for each $j$,
\begin{equation*}
\frac{\text{d}^{j}(B^N)^*}{\text{d}\kappa^{j}}(\kappa)\,\Bpsi
= \overline{\frac{\text{d}^{j}B^N}{\text{d}\kappa^{j}}(\overline{\kappa})\,\overline{\Bpsi}}.
\end{equation*}
Consequently, $G((B^N)^*,\overline{\kappa^N}) = \overline{G(B^N,\kappa^N})$ and the quantities $d_n$ and $d_n^*$ in Theorem \ref{thm:Karma3} satisfy
\begin{equation*}
d_n^* = \max_{\substack{\Bvarphi\in G((B^N)^*,\overline{\kappa^N}),\\\Vert \Bvarphi\Vert_{\BX} = 1}}\text{dist}(\Bvarphi,q_n^*\BX_n)
= \max_{\substack{\overline{\Bvarphi}\in G(B^N,\kappa^N),\\\Vert \Bvarphi\Vert_{\BX} = 1}}\text{dist}(\Bvarphi,r_n\BX_n)=d_n.
\end{equation*}
Therefore,
\begin{equation*}
d_n^* = d_n 
=\max_{\substack{\Bvarphi\in G(B^N,\kappa^N),\\\Vert \Bvarphi\Vert_{\BX} = 1}}\min_{\Bvarphi_n\in\BX_n}\Vert \Bvarphi - \Bvarphi_n\Vert_{\BX} .
\end{equation*}
By Theorem \ref{thm:Karma3} and approximation property
of edge element 
(Theorem 5.41 and Theorem 8.19 of \cite{Monk03}),
\begin{equation*}
|\kappa^N_n - \kappa^N| \leq C_2 h_n^{2t/r^N}.
\end{equation*}
\end{proof}

\section{Numerical Examples}\label{Sec4}

We present numerical examples to demonstrate the theory and the effectiveness of the proposed method. The computation is carried out using the linear edge element on a series of regular tetrahedral meshes $\{\calM_{h_l}\}_l$ for $\Omega$. The curved boundary $\Gamma_R$ (and possibly $\Gamma_D$) is approximately covered by triangles induced by the meshes. The ``variational crimes" of not using the exact computational domain or exact quadratures are admissible and do not spoil convergence rate. For discussions along this line, we refer the readers to, e.g., \cite{Dubios1990SINUM}, Chp.~10 of \cite{BrennerScott2008}, and Section~8.3 of \cite{Monk03}.

For a given tetrahedral mesh, let $\Bphi^j_n, j=1,2,\cdots,J,$ be the basis functions. To obtain the matrix form of \eqref{eq:eig-XNn}, let $S$ be the stiffness matrix with $S^{ij}=(\curl\Bphi^j_n,\curl\Bphi^i_n)$, $M$ be the mass matrix with $M^{ij}=(\Bphi^j_n,\Bphi^i_n)$, and $E(\kappa)$ be the matrix related to the truncated Calder\'on operator with $E^{ij}(\kappa)=\langle\calT^N(\kappa)\Bphi^j_n,\Bphi^i_n\rangle_{\Gamma_R}$. For all the examples, a fixed truncation order $N = 10$ is employed. Numerical results indicate that $10$ is enough for the computed poles in this paper (relatively small in norm). Increase of $N$ does not improve the accuracy significantly. However, it is expected that larger $N$'s and finer meshes are necessary for larger poles.  Note that the computer available to use limits the problem size. 

The nonlinear algebraic eigenvalue problem for \eqref{eq:eig-XNn} is to find $\kappa\in\bbC$ and $\mathbf{u} \ne \mathbf{0}$ such that
\begin{equation}
\label{eq:eig-F}
F(\kappa)\mathbf{u} = \mathbf{0}
\end{equation}
with 
\begin{equation*}
F(\kappa) = S - \kappa^2 M - \rmi\kappa E(\kappa).
\end{equation*}

For all examples, we employ the spectral indicator method (SIM) to compute the eigenvalues of \eqref{eq:eig-F} in a region $\Theta$ given by
\begin{equation*}
\Theta:=\{a+b\rmi\in\bbC: a\in[0,2),\, b\in(-2,0) \}.
\end{equation*}
SIM is based on contour integrals and was originally designed for non-Hermitian matrix eigenvalue problems \cite{Huang2016JCP, Huang2018NLAA}. The parallel SIM is highly scalable and has been successfully used to compute several nonlinear eigenvalue problems \cite{Gong2022MC, MaSun2023, XiLinSun2024CMA}. Unlike Newton-type or gradient descent methods, the parallel SIM does not require initial guesses for the resonance locations, which is particularly important given the complexity of the spectrum.

\subsection{Unit Ball}

Let the obstacle $D$ be the unit ball centered at the origin. For the PEC boundary condition \eqref{eq:e2}, the resonances are $\kappa$'s such that
\begin{equation}
\label{eq:exactKappa}
\{\kappa\in\bbC: \, d_n(\kappa) =h^{(1)}_n(\kappa)z^{(1)}_n(\kappa)= 0,\, n\in\bbN\}.
\end{equation}
Figure~\ref{Fig:Exact} displays the exact resonances in  $\Theta$. For each positive integer $n$, a resonance $\kappa$ satisfying $d_n(\kappa)=0$ has multiplicity $2n+1$, corresponding to the azimuthal indices $m=-n,\cdots,n$. The resonances $\kappa_1$ and $\kappa_2$ have multiplicity 3, $\kappa_3$, $\kappa_4$ and $\kappa_5$ have multiplicity 5, and $\kappa_6$ has multiplicity 7.

\begin{figure}[htbp]
	\centering
	\includegraphics[width=0.8\textwidth]{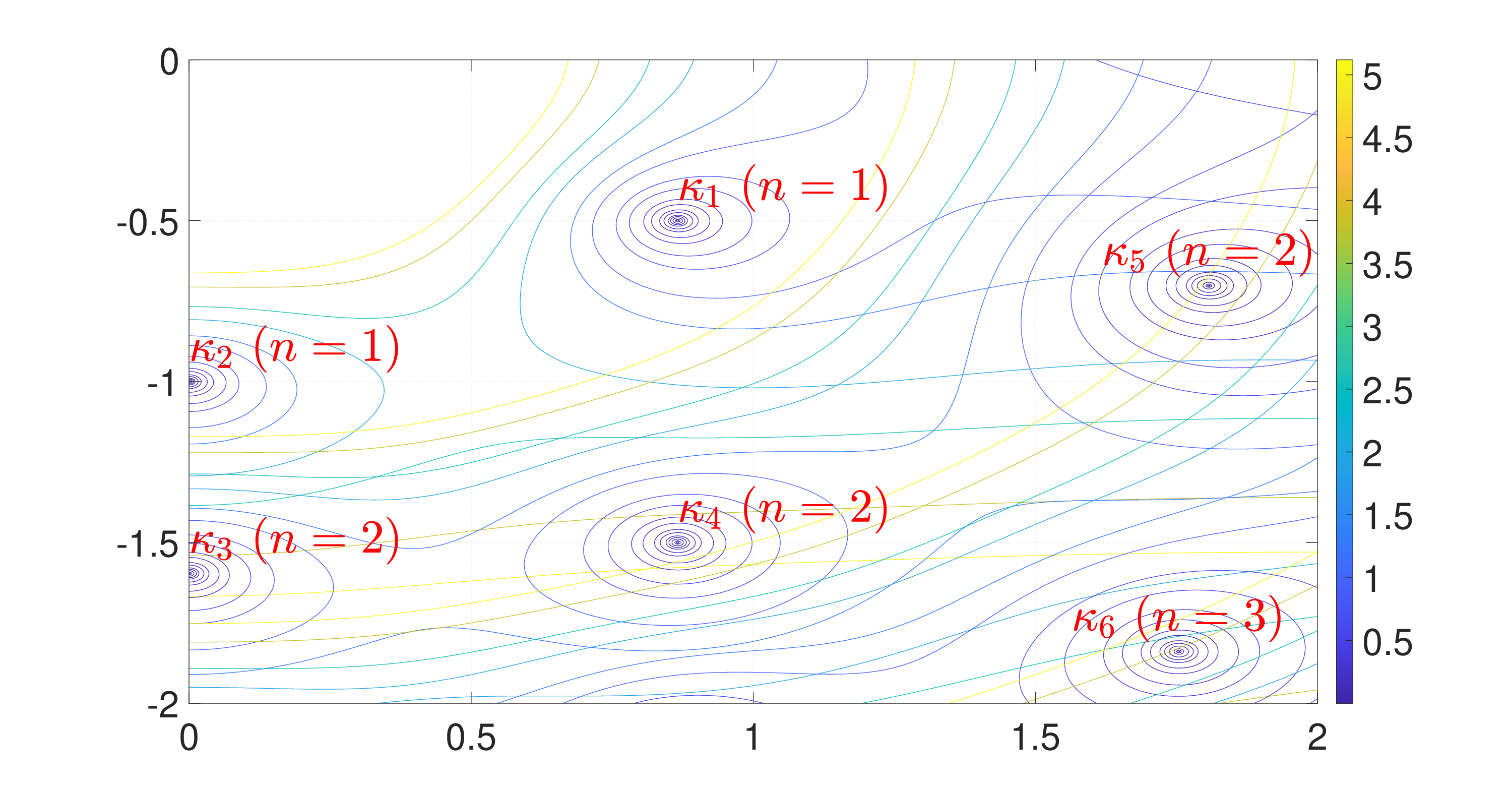} 
	\caption{Poles for the unit ball (zeros of $d_n(\kappa)$).}
	\label{Fig:Exact}
\end{figure}

We truncate the exterior domain by an artificial spherical boundary of radius $R=1.3$ centered at the origin. The computed resonances in $\Theta$ are shown in Figure \ref{Fig:UnitBall}, which are consistent with the exact values. Note that when $\partial D$ is a sphere, these resonances are related to Mie resonances \cite{Bohren1983}.

\begin{figure}[htbp]
	\centering
	\includegraphics[width=0.8\textwidth]{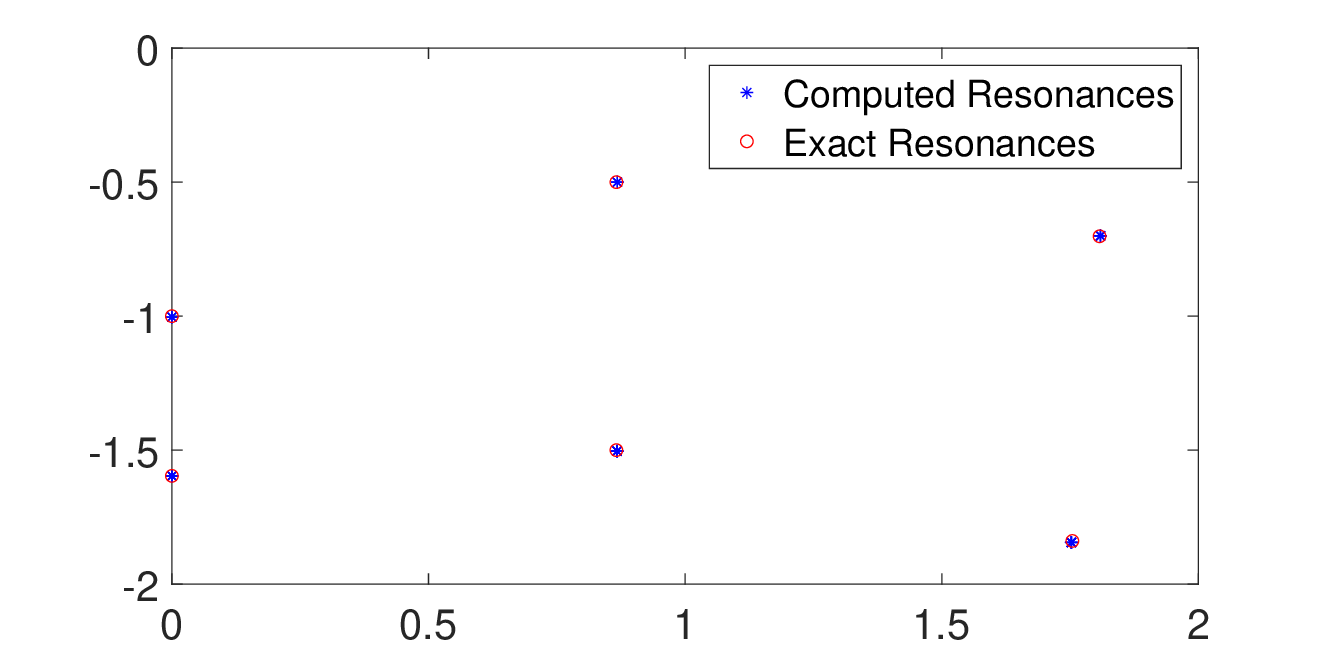}
	\caption{Computed resonances for the unit ball.}
	\label{Fig:UnitBall}
\end{figure}

Let $\kappa_{h_l}$ be the resonance computed on the mesh $\calM_{h_l}$ and $\kappa$ be the exact resonance obtained from \eqref{eq:exactKappa}. The convergence order is defined as
\begin{equation*}
r_l = -\frac{\log\left(\abs{\kappa_{h_l} - \kappa} / \abs{\kappa_{h_{l-1}} - \kappa}\right)}{\log(N_l/ N_{l-1})^{1/3}},
\end{equation*}
where $N_l$ is the number of degrees of freedom (DOFs) for mesh $\calM_{h_l}$. To facilitate comparison, we average the $J_i$ numerically clustered resonances $\kappa^j_{i,h_l}$ that approximate $\kappa_i$ by defining $\hat\kappa_{i,h_l}=\frac{1}{J_i}\sum_{j=1}^{J_i}\kappa^j_{i,h_l}$, and then calculate the error $|\hat\kappa_{i,h_l} -\kappa_i|$. In Table \ref{Tab:UnitBall}, we present the average of  computed resonances and the convergence orders for the two resonances with the smallest absolute values. It is observed that, for the unit ball, the convergence orders of the first two resonances are approximately two.

\begin{table}[h]
	\begin{center}
		\resizebox{\linewidth}{!}{
			\begin{tabular}{rcccc}
				\hline
				$N_l$&$\hat\kappa_{1,h_l}$&Ord&$\hat\kappa_{2,h_l}$&Ord\\
				\hline
				2703&0.875950 - 0.502153i&&0.867324 - 1.515403i&\\
				\hline
				5484&0.873037 - 0.501016i&1.53&0.869709 - 1.512957i&0.58\\
				\hline
				10758&0.870205 - 0.501296i&2.15&0.867687 - 1.506936i&2.83\\
				\hline
				21624&0.868509 - 0.500110i&2.43&0.866751 - 1.504777i&1.67\\
				\hline
				43872&0.867836 - 0.499886i&1.34&0.867354 - 1.503993i&0.59\\
				\hline
				86064&0.867132 - 0.500057i&2.20&0.866770 - 1.502301i&2.47\\
				\hline
			\end{tabular}
		}
		\caption{Computed resonances 
			and convergence orders for the unit ball.}
		\label{Tab:UnitBall}
	\end{center}
\end{table}

\subsection{Unit Cube}

Let the obstacle be the unit cube $D=(-\frac12,\frac12)^3$ and $R=1$. Six resonances are found in $\Theta$. The values on the finest mesh with $N_l=130928$  DOFs are 
\begin{align*}
	&0.000000 - 1.556369i,\ 
	&&0.000000 - 1.556234i, \
	&&0.000000 - 1.556268i, \\
	&1.315827 - 0.715144i, \
	&&1.315960 - 0.715407i, \
	&&1.316126 - 0.715631i.
\end{align*}

In this case, the exact resonances are unknown. We use the relative errors and define the convergence order as
\begin{equation*}
r_l = -\frac{\log\left(\abs{\kappa_{h_l}-\kappa_{h_{l-1}}} / \abs{\kappa_{h_{l-1}}-\kappa_{h_{l-2}}}\right)}{\log (N_l / N_{l-1})^{1/3}}.
\end{equation*}
Table~\ref{Tab:UnitCube} lists the computed resonances and the convergence orders for the two resonances with the smallest absolute values. 

\begin{table}[h]
	\begin{center}
		\resizebox{\linewidth}{!}{
			\begin{tabular}{rcccc}
				\hline
				$N_l$&$\hat\kappa_{1,h_l}$&Ord&$\hat\kappa_{2,h_l}$&Ord\\
				\hline
				8214&0.000000 - 1.542593i&&1.296219 - 0.697389i&\\
				\hline
				16366&0.000000 - 1.547902i&&1.302799 - 0.703003i&\\
				\hline
				32976&0.000000 - 1.551898i&1.22&1.309138 - 0.709054i&-0.05\\
				\hline
				65712&0.000000 - 1.554638i&1.64&1.313384 - 0.712724i&1.94\\
				\hline
				130928&0.000000 - 1.556290i&2.20&1.315971 - 0.715394i&1.79\\
				\hline
			\end{tabular}
		}
		\caption{Computed resonances and convergence orders for the unit cube.}
        \label{Tab:UnitCube}
	\end{center}
\end{table}

\subsection{L-shaped Domain}

Let the obstacle be the L-shaped domain $D=(-\frac12, \frac12)^3 \backslash[0, \frac12]^3$ and $R=1$. Six resonances are found in $\Theta$. The values on the finest mesh with $N_l=132624$  DOFs are
\begin{align*}
&0.000000 - 1.627763i,\
&&0.000000 - 1.627865i, \
&&0.000000 - 1.563102i, \\
&1.377776 - 0.805608i, \
&&1.336973 - 0.680441i, \
&&1.337279 - 0.680878i.
\end{align*}

Table \ref{Tab:L-shaped} lists the computed resonances and the convergence orders for the two resonances with the smallest absolute values.

\begin{table}[h]
	\begin{center}
		\resizebox{\linewidth}{!}{
			\begin{tabular}{rcccc}
				\hline
				$N_l$&$\hat\kappa_{1,h_l}$&Ord&$\hat\kappa_{2,h_l}$&Ord\\
				\hline
				8203&0.000000 - 1.546757i&&1.313144 - 0.662566i&\\
				\hline
				16578&0.000000 - 1.553662i&&1.322451 - 0.669346i&\\
				\hline
				33740&0.000000 - 1.557846i&2.12&1.329060 - 0.674856i&1.23\\
				\hline
				65624&0.000000 - 1.560786i&1.59&1.333465 - 0.677605i&2.28\\
				\hline
				132624&0.000000 - 1.563102i&1.02&1.337126 - 0.680660i&0.36\\
				\hline
			\end{tabular}
		}
		\caption{Computed resonances and convergence orders for the L-shaped domain.}
		\label{Tab:L-shaped}
	\end{center}
\end{table}

\subsection{Cavity}

Let the obstacle be $D=[-1.2,1.2]^3\backslash\left([-1,1]^3\cup [1,1.2]\times[-0.2,0.2]^2\right)$ and $R=2.3$ (see Section 5.3 of \cite{Nannen2013SISC}). The values on the finest mesh with $N_l=127480$ DOFs are
\begin{align*}
	&0.549073 - 0.299012i,\ 
	&&0.549002 - 0.298857i, \
	&&0.549044 - 0.298929i, \\
	&1.075874 - 0.307260i, \
	&&1.075719 - 0.307112i, \
	&&1.075617 - 0.307005i.
\end{align*}

Table~\ref{Tab:Cavity} lists the computed resonances and the convergence orders for the two resonances with the smallest absolute values. 

\begin{table}[h]
	\begin{center}
		\resizebox{\linewidth}{!}{
			\begin{tabular}{rcccc}
				\hline
				$N_l$&$\hat\kappa_{1,h_l}$&Ord&$\hat\kappa_{2,h_l}$&Ord\\
				\hline
				8328&0.542518 - 0.293289i&&1.053715 - 0.297677i&\\
				\hline
				15935&0.544389 - 0.294788i&&1.059832 - 0.300287i&\\
				\hline
				32405&0.546542 - 0.296474i&-0.55&1.067209 - 0.302659i&-0.65\\
				\hline
				66624&0.548347 - 0.298169i&0.41&1.073125 - 0.305457i&0.70\\
				\hline
				127480&0.549040 - 0.298933i&4.05&1.075737 - 0.307126i&3.46\\
				\hline
			\end{tabular}
		}
		\caption{Computed resonances and convergence orders for the cavity.}
		\label{Tab:Cavity}
	\end{center}
\end{table}

\section*{Acknowledgement}
The research of B. Gong is partially supported by National Natural Science Foundation of China No.12201019. The research of J. Sun is partially supported by an NSF Grant DMS-2109949 and a SIMONS Foundation Grant 711922. The research of X. Wu is partially supported by National Key R\&D Program of China 2019YFA0709600, 2019YFA0709601.

\end{document}